\documentclass[11pt, leqno]{article}
\usepackage{amsmath,amsfonts,amsthm,amssymb,amscd}

\newtheorem{thm}{Theorem}[section]
\newtheorem{cor}{Corollary}[section]
\newtheorem{prop}{Proposition}[section]
\newtheorem{lem}{Lemma}[section]
\newtheorem{sublem}{Sublemma}[section]
\newtheorem{clm}{Claim}

\theoremstyle{definition}
\newtheorem{defi}{Definition}[section]
\newtheorem{as}{Assumption}[section]
\newtheorem{ex}[thm]{Example}

\theoremstyle{remark}
\newtheorem{rem}{Remark}[section]

\numberwithin{equation}{section}

\renewcommand{\t}{\tau}
\newcommand{\tilV}{\tilde{V}}
\newcommand{\tilNu}{\mathcal{\widetilde V}}
\newcommand{\vol}{{\rm Vol}}
\newcommand{\dert}[1]{\frac{\partial #1}{\partial t}}
\newcommand{\dertau}[1]{{\partial #1\over\partial\tau}}

\newcommand{\Ric}{{\rm Ric}}
\newcommand{\Rm}{{\rm Rm}}
\newcommand{\Hess}{{\rm Hess}\,}
\newcommand{\tr}{{\rm tr}}
\renewcommand{\L}{\mathcal{L}}
\renewcommand{\l}{\ell}
\renewcommand{\H}{\mathcal{H}}
\newcommand{\Z}{\mathbb{Z}^{+}}
\newcommand{\R}{\mathbb{R}}
\newcommand{\K}{\mathcal{K}}

\renewcommand{\[}{\Bigl[}
\renewcommand{\]}{\Bigr]}
\renewcommand{\(}{\Bigl(}
\renewcommand{\)}{\Bigr)}
\renewcommand{\|}{\Big|}

\renewcommand{\phi}{\varphi}
\renewcommand{\epsilon}{\varepsilon}

\newcommand{\bt}{{\bar{\tau}}}
\newcommand{\btt}{\bar{t}}
\newcommand{\bs}{\bar{\sigma}}
\newcommand{\bg}{\bar{g}}
\newcommand{\bp}{\bar{p}}
\newcommand{\bq}{\bar{q}}
\newcommand{\bM}{\bar{M}}

\begin{document}

\title{Perelman's reduced volume and a gap theorem for~the~Ricci flow}

\author{Takumi Yokota}

\date{}

\maketitle

\begin{abstract}
In this paper, we show that any ancient solution to the
Ricci flow with the reduced volume whose asymptotic limit
is sufficiently close to that of the Gaussian soliton is
isometric to the Euclidean space for all time. This is a
generalization of Anderson's result for Ricci-flat
manifolds. As a corollary, a gap theorem for gradient
shrinking Ricci solitons is also obtained.
\end{abstract}

\section{Introduction}\label{sec1}

Let us consider a smooth one-parameter family of Riemannian
metrics $g(t),\break t \in [0, T)$ on a manifold $M$. We call
$(M, g(t))$ a {\it Ricci flow} if it satisfies
\begin{equation}
\dert{}g = -2\Ric
\end{equation}
where $\Ric$ denotes the Ricci tensor of $g(t)$. We also
use $R:=\tr \Ric$ to denote the scalar curvature. The
purpose of the present paper is to show a gap theorem for
the Ricci flow. In order to state our main theorem, we
first recall a heuristic argument given
in~\cite[Section~6]{P}.

In his seminal paper~\cite{P}, Perelman introduced a comparison
geometric approach to the Ricci flow, called reduced geometry
in~\cite{Ni-Mean}. For a Ricci flow $(M^n, g(t)), t\in [0, T)$
with singular time $T$, take $T_0 < T$ and consider the backward
Ricci flow $g(\t)$, where $\t :=T_0-t \in [0, T_0]$ is the
reverse time. Equipping $\widetilde{M} :=M\times S^N\times (0,
T_0]$, for large $N\gg 1$, with a metric $\widetilde{g}$\break
written as
\begin{equation}
\widetilde{g} = g(\t) + \t g_{S^N} +\left(
R+\frac{N}{2\tau}\right)\,d\t^2
\end{equation}
Perelman observed that $(\widetilde{M}, \widetilde{g})$ has
vanishing Ricci curvature up to mod\,$N^{-1}$. Here, $(S^N,
g_{S^N})$ is the $N$-sphere with constant curvature
$\frac{1}{2N}$. An easy way to get a feeling of this is to
regard $\widetilde{g}$ as a cone metric by setting $\eta :=
\sqrt{2N\t}$. Recall that the metric cone $(N\times(0, T),
d\eta^2+\eta^2g_N)$ of $(N, g_N)$ is Ricci-flat if and only
if $\Ric_{g_N}= ({\rm dim}N-1)g_N$. Then he applied the
Bishop-Gromov inequality to $(\widetilde{M},
\widetilde{g})$ formally to obtain an invariant $\tilV_{(p,
0)}(\t)$ which he called the reduced volume. As expected,
the reduced volume is non-increasing in $\t$ (Theorem~2.1)
and his first application of this was the (re)proof of his
no local collapsing theorem~\cite[Section~7]{P}.

Throughout this paper, we adopt the convention that the
reduced volume is identically 1 for the Gaussian soliton.
The {\it Gaussian soliton} is the trivial Ricci flow
$(\R^n, g_{\rm E})$ on the Euclidean space regarded as a gradient
shrinking Ricci soliton $(\R^n, g_{\rm E}, \frac{|\,\cdot\,|^2}{4})$.

Now we state our main theorem of this paper.
\begin{thm}\label{main}
There exists $\epsilon_n>0$ which depends only on $n\ge2$
and\break satisfies the following: let $(M^n, g(\t)), \t\in [0,
\infty)$ be a complete ancient solution to the Ricci flow on an
$n$-manifold $M$ with Ricci curvature bounded below. Suppose
that the asymptotic limit of the reduced volume $\lim_{\t
\rightarrow \infty}\tilV_{(p, 0)}\break (\t)$ is greater than
$1-\epsilon_n$ for some $p \in M$. Then $(M^n, g(\t))$ is the
Gaussian soliton, i.e., isometric to the Euclidean space
$(\mathbb{R}^n, g_{\rm E})$ for all $\t\in [0, \infty)$.
\end{thm}

We say that $(M, g(\t))$ is {\it ancient} when $g(\t)$
exists for all $\t \in[0, \infty)$. Ancient solutions are
important objects in the study of singularities of the
Ricci flow. The limit $\tilNu(g) :=
\lim_{\t\to\infty}\tilV_{(p, 0)}(\t)$ will be called the
{\it asymptotic reduced volume} of the flow $g(\t)$. We
will see in Lemma~3.1 below that $\tilNu(g)$ is independent
of the choice of $p\in M$.

By regarding a Ricci-flat metric as an ancient solution as in
Theorem \ref{main}, we recover the following result, which is
the motivation of the present\break paper.
\begin{thm}[{\rm \cite[Gap Lemma 3.1]{A}}]\label{gap}
There exists $\epsilon_n>0$ which satisfies the following:
let $(M^n, g)$ be an $n$-dimensional complete Ricci-flat
Riemannian manifold. Suppose that the asymptotic volume
ratio $\nu(g) := \lim_{r \to\infty}\break \vol\,B(p
,r)/\omega_n r^n$ of $g$ is greater than $1-\epsilon_n$.
Here $\omega_n$ stands for the volume of the unit ball in
the Euclidean space $(\mathbb{R}^n, g_{\rm E})$. Then $(M^n, g)$
is isometric to $(\mathbb{R}^n, g_{\rm E})$.
\end{thm}

On the way to the proof of Theorem~\ref{main}, we establish
several lemmas. Here we state one of them as a theorem,
which is of independent interest.
\begin{thm}\label{finitefundaRicci}
Let $(M^n, g(\t)), \t\in[0, \infty)$ be a complete ancient
solution to the Ricci flow on $M$ with Ricci curvature bounded
below. If $\tilNu(g)>0$, then~the fundamental group of $M$ is
finite. In particular, any ancient $\kappa$-solution to the
Ricci flow has finite fundamental group.
\end{thm}
More generally, Theorem~\ref{finitefundaRicci} is shown for
super Ricci flows in Lemma~\ref{funda} under certain
assumptions. See also Remark~\ref{rem:funda} below for
application.

Finally, we apply the theorems above to gradient shrinkers.
We call a triple $(M, g, f)$ a {\it gradient shrinking Ricci soliton} when
\begin{equation*}
\Ric+ \Hess\, f -\frac{1}{2\lambda}g=0
\end{equation*}
holds for some positive constant $\lambda>0$. Shrinking
Ricci solitons are typical examples of ancient solutions to
the Ricci flow. We normalize the potential function $f \in
C^\infty(M)$ by adding a constant so that
\begin{equation}\label{fnormalize}
R+ |\nabla f|^2 -\frac{f}{\lambda}=0\quad \hbox{on } M.
\end{equation}
The left-hand side of $(\ref{fnormalize})$ is known to be
constant~\cite[Proposition 1.15]{V2}.

\begin{cor}\label{main3}
Let $(M^n, g, f)$ be a complete gradient shrinking Ricci
soliton with Ricci curvature bounded below. Then
\begin{enumerate}
\item[$(1)$]the fundamental group of $M$ is finite and
\item[$(2)$]the normalized $f$-volume $\int_M(4\pi\lambda)^{-n/2}{\rm e}^{-f}\,d\mu_g$ does not exceed $1$.
\item[$(3)$]Suppose that
\begin{equation*}
\int_M(4\pi\lambda)^{-n/2}e^{-f}d\mu_g >
1-\epsilon_n,
\end{equation*}
then $(M^n, g, f)$ is, up to scaling, the Gaussian soliton
$(\R^n, g_{\rm E}, \frac{|\,\cdot\,|^2}{4})$. Here the
constant $\epsilon_n$ comes from Theorem~$\ref{main}$.
\end{enumerate}
\end{cor}

Part (1) of Corollary~\ref{main3} is a restatement of the
result obtained by many people in more general context
(cf.~\cite{W}). The other statements in
Corollary~\ref{main3} are intimately related to the results
of Carrillo--Ni~\cite{CN}. In particular,
Corollary~1.2.(3) proves their conjecture that {\it
the normalized $f$-volume is $1$ only for the Gaussian
soliton}~\cite{CN}. See Remark~\ref{finalrem} below.

The paper is organized as follows. In Section 2, we review
definitions and Perelman's results in~\cite{P}. We will do
this for super Ricci flows. In Section~3, we prove some
lemmas required in the proof of the main theorem. In
Section 4, we give a proof of Theorem~\ref{main}. In
Section 5, we prove Corollary~\ref{main3} and consider
expanding solitons with non-negative Ricci curvature. The
final section contains some remarks. Appendix A is devoted
to detailed proofs of the facts used in the argument
without proof.

\section{Comparison geometry of super Ricci flows}

\subsection{Super Ricci flow}

In this section, we recall the
definitions and results in \cite[Sections~6 and 7]{P}. The
main references are \cite{P,Ye, KL, V2}. Among them,
Ye~\cite{Ye} paid careful attention to argue under the
assumption of Ricci curvature bounded below rather than
bounded sectional curvature (see also
\cite[Appendix]{EKNT}). The assumption of Theorem
\ref{main} on the Ricci flow $(M^n, g(\t))$ is the same as
that considered in~\cite{Ye}. We mainly adopt the notation
of~\cite{V2}.

We would like to develop Perelman's reduced geometry in
more general situation, that is, the super Ricci flow. This
will provide us with a convenient setting for comparison
geometry of the Ricci flow. A smooth one-parameter family
of Riemannian metrics $(M, g(\t)), \t\in [0, T)$ is called
a {\it super Ricci flow} when it satisfies
\begin{equation}
\dertau{}g \le 2\,\Ric.
\end{equation}
Super Ricci flow was introduced by
McCann--Topping~\cite{MT} in their attempt to generalize
the contraction property of heat equation in the
Wasserstein spaces, which characterizes the non-negativity
of the Ricci curvature of the Riemannian metrics
(see~\cite{RS}), to time-depending metrics. See
also~\cite{Lo} for this topic.

Basic and important examples of super Ricci flows are
\begin{ex}\label{ex}
\begin{enumerate}
\item[(1)]A solution to the backward Ricci flow equation $\dertau{}g=2\,\Ric$ and
\item[(2)]$g(\t) := (1+2C\t)g_0, \t \in[0, \frac{1}{|C|-C})$ for some fixed Riemannian metric\vspace*{3pt} $g_0$ with Ricci curvature bounded from below by $C \in \R$.
\end{enumerate}
\end{ex}
Therefore, it can be said that the study of super Ricci
flows includes those of (backward) Ricci flows and
manifolds with Ricci curvature bounded from below.

We can straightforwardly generalize Perelman's reduced geometry to the super Ricci flow
if we impose the following assumptions.
\begin{as}\label{as}
Putting $2h := \dertau{}g$ and $H := \tr_{g(\t)} h$, $h$ satisfies
\begin{enumerate}
\item[(1)]contracted second Bianchi identity $2\,{\rm div}\,h(\cdot) = \langle \nabla H, \cdot\rangle$ and
\item[(2)]heat-like equation $-\tr_{g(\t)}\dertau{}h \ge \varDelta_{g(\t)} H$, or equivalently,
\begin{equation}\label{evol}
-\dertau{}H \ge \varDelta_{g(\t)} H + 2|h|^2.
\end{equation}
\end{enumerate}
\end{as}
Clearly, the ones in Example~\ref{ex} above satisfy
Assumption \ref{as}. It is known that the evolution
equation for the scalar curvature $R$ under the Ricci flow
$g(\t)$ is given by $-\dertau{}R= \varDelta_{g(\t)} R +
2|\Ric|^2.$

In what follows, we denote by $(M, g(\t)), \t\in [0, T)$ a
complete super, or backward Ricci flow on an $n$-manifold
$M$ satisfying Assumption~\ref{as}. It is also assumed that
the time-derivative $\dertau{}g$ is bounded from below in
each compact time interval, that is, for any compact
interval $[\t_1, \t_2] \subset [0, T)$, we can find
$K=K(\t_1, \t_2) \ge 0$ such that $-Kg(\t) \le \dertau{}g
\le 2\,\Ric_{g(\t)}$ and hence
\begin{equation*}
\hbox{e}^{K(\t_2-\t)}g(\t_2) \ge g(\t) \ge
\hbox{e}^{-K(\t-\t_1)}g(\t_1)
\end{equation*}
for all $\t\kern-1pt \in\kern-1pt [\t_1, \t_2]$. Although Assumption~\ref{as}
looks too restrictive, the author's intention is a unified
treatment of backward Ricci flows and Riemannian manifolds
with non-negative Ricci curvature. (See also
Remark~\ref{rem:entropy}~below.)

\subsection{Definition of the reduced volume}

Let us start with the definitions.
Fix $p \in M$, $[\t_1, \t_2] \subset [0, T)$ and\break $\bt \in
(0, T)$.
\begin{defi}
Let $\gamma : [\t_1, \t_2] \rightarrow M$ be a curve.
We define the {\it $\L$-length} of $\gamma$ and the {\it $\L$-distance}, respectively, by
\begin{equation*}
\L(\gamma) := \int_{\t_1}^{\t_2} \sqrt\tau \( \|\frac{d\gamma}{d\t}\|^2_{g(\t)} + H(\gamma(\t), \t)\) d\t
\end{equation*}
and
\begin{equation*}
L_{(p, \t_1)}(q, \t_2) := \inf \bigl\{\L(\gamma) ;\ \gamma
: [\t_1, \t_2]\rightarrow M \hbox{ with } \gamma(\t_1)=p,
\gamma(\t_2)=q \bigr\}.
\end{equation*}
\end{defi}

The lower bound of $\dertau{}g$ guarantees that the
$\L$-distance between any two points is achieved by a
minimal $\L$-geodesic. This is the only place where we
employ the assumption on $\dertau{}g$. A curve $\gamma(\t)$
is called an {\it $\L$-geodesic} when
\begin{equation}\label{geod}
2\nabla_X X + \frac{X}{\t} - \nabla H + 4h(X, \cdot) =0, \quad X:=\frac{d\gamma}{d\t}(\t)
\end{equation}
is satisfied.

Then the {\it reduced distance} and the {\it reduced
volume} based at $(p, 0)$ are defined, respectively, by
\begin{equation*}
\l_{(p, 0)}(q, \bt):= \frac{1}{2\sqrt{\bt}} L_{(p, 0)}(q, \bt)
\end{equation*}
and
\begin{equation*}
\tilV_{(p, 0)}(\bt):= \int_M
(4\pi\bt)^{-n/2}\hbox{e}^{-\l_{(p, 0)} (q,
\bt)}d\mu_{g(\bt)}(q)
\end{equation*}
where $d\mu_{g(\bt)}$ denotes the volume element induced by
$g(\bt)$.

We can rewrite the reduced volume as
\begin{align}\label{eq:reducedpullback}
\tilV_{(p, 0)}(\bt) = \int_{T_pM}
(4\pi\bt)^{-n/2}\exp\(-\l_{(p, 0)}(\L\,\exp_{\bt}(V),\,
\bt)\)\L J_V(\bt)dx_{g(0)}(V)
\end{align}
\looseness=-1 by pulling back the integrand by the {\it
$\L$-exponential map} $\L\,\exp_{\bt} : T_pM \to M$ which
assigns $\gamma_V(\bt)$, if exists, to each $V \in T_pM$. Here
$\gamma_V$ is the $\L$-geodesic determined by $\gamma_V(0)=p$
and $\lim_{\t \to 0+}\sqrt\t \frac{d\gamma}{d\t}(\t)=V$. In
(\ref{eq:reducedpullback}), $dx_{g(0)}$ denotes the Lebesgue
measure on $T_pM$ induced by the metric $g(0)$ and $\L J_V(\bt)$
is called the {\it $\L$-Jacobian}. Remember that we are using
the convention that $\L J_V(\bt)=0$ unless $V\in \Omega_{(p,
0)}(\bt)$. By $V \in \Omega_{(p, 0)}(\bt)$, we mean that
$\L\,\exp_{\bt}(V)$ exists and lies outside the {\it $\L$-cut
locus} at time $\bt$. It follows that $\Omega_{(p, 0)}(\bt)$ is
an open set of $T_pM$, on which $\L\,\exp_{\bt}$ is a
diffeomorphism, and that $\Omega_{(p, 0)}(\t_2) \subset
\Omega_{(p, 0)}(\t_1)$ for $\t_2 > \t_1>0$. The base point $(p,
0)$ will often be suppressed.\vskip6pt

\subsection{Monotonicity of the reduced volume}

Next, we recall the
computations performed in \cite[Section~7]{P}.\vskip6pt

Let $q \in \L\,\exp_{\bt}(\Omega_{(p, 0)}(\bt))$ and $\gamma :
[0, \bt] \rightarrow M$ be the unique minimal\break \hbox{$\L$-geodesic} from
$(p, 0)$ to $(q, \bt)$. Take a tangent vector $Y \in T_qM$ and
extend it to the vector field along $\gamma$ by
solving\vspace*{-3pt}
\begin{equation*}
\nabla_X Y = -h(Y, \cdot) + \frac{Y}{2\t}, \qquad
Y(\bt)=Y\vspace*{-3pt}
\end{equation*}
so that $|Y|^2(\t) = \frac{\t}{\bt}|Y|^2$.

Then we have that $\nabla \l(q, \bt) = \frac{d\gamma}{d\t}(\bt)$
and\vspace*{-3pt}
\begin{align}
\dertau{}\l(q, \bt) &= H(q, \bt) -\frac{\l(q, \bt)}{\bt} + \frac{1}{2\bt^{3/2}}K\nonumber\\
|\nabla \l|^2(q, \bt) &= -H(q, \bt) +\frac{\l(q, \bt)}{\bt} -\frac{1}{\bt^{3/2}}K\label{Hess}\\
\Hess\,\l(Y, Y)(q, \bt) &\le -h(Y, Y) +
\frac{|Y|_{g(\bt)}^2}{2\bt} - \frac{1}{2\sqrt{\bt}}
\int_0^{\bt} \sqrt{\t} \H(X, Y)\,d\t\nonumber\\
\Delta \l(q, \bt) &\le -H(q, \bt)+ \frac{n}{2\bt} - \frac{1}{2\bt^{3/2}}K\label{Laplace}\\
\dertau{}\log\,\L J_V(\bt) &= \Delta \l(q, \bt) +H(q, \bt)
\le \frac{n}{2\bt} -
\frac{1}{2\bt^{3/2}}K.\nonumber\vspace*{-3pt}
\end{align}
Here, following \cite[Section~7]{P}, we have put
\begin{align*}
\H(X) &:= -\dertau{H} -\frac{H}{\tau} -2\langle\nabla H, X\rangle +2h(X, X)\\
K &:= \int_0^{\bt} \t^{3/2}\H(X)d\t\\
\H(X, Y) &:= - \langle\nabla_Y\nabla H, Y\rangle +2\langle R(X, Y)Y, X\rangle +4\nabla_Yh(X, Y) -4\nabla_Xh(Y, Y)\\
&\quad\hphantom{:}-2\dertau{h}(Y, Y) +2|h(Y, \cdot)|^2 -
\frac{1}{\t}h(Y, Y).
\end{align*}

The point where we have used Assumption~\ref{as} is the
derivation of (\ref{Laplace}) from (\ref{Hess})
(cf.~\cite[Lemma 7.42)]{V2}):
\begin{align*}
&\tr\,\H(X, \cdot)\\
&\quad= -\Delta H +2\,\Ric(X, X)+4\,{\rm div}\,h(X)-4\langle\nabla H, X\rangle
-2\dertau{H}-2|h|^2 -\frac{H}{\t}\\
&\quad= \H(X) + 2\[\Ric(X, X)-h(X, X)\]+\[-\dertau{H}-\varDelta H-2|h|^2\]\\ &\qquad + 2\[2\,{\rm div}\,h(X)- \langle\nabla H, X\rangle\]\\ &\quad\ge \H(X).
\end{align*}
The quantities corresponding to $\H(X)$ and $\tr\,\H(X,
\cdot)$ appear in~\cite[(1.2)]{ChCh} and \cite[(1.4)]{ChCh}
as the trace Harnack expressions of Hamilton~\cite{Ha-Ha}
and Chow--Hamilton \cite{CH}, respectively.

We now state the main theorem of this section (cf. [8, 25, 29]).
\setcounter{thm}{0}
\begin{thm}\label{reduced}
Let $(M^n, g(\t)), \t \in [0, T)$ be a complete super Ricci
flow satisfying~Assumption $\ref{as}$ with time derivative
bounded below. Then for any $p\in M$ and $V\in T_pM$,
\begin{equation}\label{integrand}
(4\pi\t)^{-n/2}{\rm e}^{-\l_{(p, 0)}(\gamma_V(\t),\, \t)}\L J_V(\t)
\end{equation}
is non-increasing in $\t$ and
\begin{equation*}
\lim_{\t\to 0+} \[(4\pi\t)^{-n/2}e^{-\l_{(p,
0)}(\gamma_V(\t),\, \t)}\L J_V(\t)\] =
\pi^{-n/2}e^{-|V|_{g(0)}^2}.
\end{equation*}
Moreover, $(\ref{integrand})$ is constant on $(0, \bt]$ if
and only if the shrinking soliton equation:
\begin{equation}\label{soliton}
\[ \frac{1}{2}\dertau{g} + \Hess\,\l_{(p, 0)} -\frac{1}{2\t}g \](\gamma_V(\t), \t) =0
\end{equation}
holds along the $\L$-geodesic $\gamma_V(\t)$ for $\t \in (0, \bt]$.

Hence, $\tilV_{(p, 0)}(\t)$ is non-increasing in $\t$,
$\lim_{\t\to 0+} \tilV_{(p, 0)}(\t) =1$ and hence
$\tilV_{(p, 0)}(\t) \le 1$. Moreover, $\tilV_{(p,
0)}(\bt)=1$ for some $\bt>0$ if and only if $(M^n,\break
g(\t)), \t \in[0, \bt]$ is the Gaussian soliton.
\end{thm}

We need to give a proof that $\tilV_{(p, 0)}(\bt)=1$ for some
$\bt>0$ implies that $(M^n, g(\t))$ is the Gaussian soliton on
$[0, \bt]$. The proofs of the other statements are minor
modifications of those of~\cite[Lemma 8.16, Corollary 8.17]{V2}
for the Ricci flow. It should be noted that we have no
assumption on the curvature of $g(\t)$ other than the lower
bound of $\dertau{}g$ in contrast to\break \cite[Corollary
8.17]{V2}.

\begin{proof}[Proof of Theorem~$\ref{reduced}$]
Suppose that $\tilV_{(p, 0)}(\bt)=1$. This implies that $M$
is simply connected. Otherwise, the reduced volume of the
universal covering $(\bM, \bg(\bt))$ of $(M, g(\bt))$ must
be greater than $1$, which is a contradiction.

Fix some small $\t_\delta \in (0, \bt)$.
For any $\t \in (\t_\delta, \bt]$,
let $\phi_{\t-\t_\delta} : M \to M$ be the map which sends $q\in M$ to $\gamma(\t)$,
where $\gamma: [0, \bt]\to M$ is the minimal $\L$-geodesic passing $(q, \t_\delta)$ with $\gamma(0)=p$.

Since $\dertau{}\phi_{\t-\t_\delta}(q) = \frac{d\gamma}{d\t}(\t) = \nabla \l_{(p, 0)}(\gamma(\t), \t)$,
we deduce from (\ref{soliton}) that
\begin{equation*}
\dertau{}\frac{1}{\t}(\phi_{\t-\t_\delta})^*g(\t) =
\frac{1}{\t}(\phi_{\t-\t_\delta})^*\[-\frac{1}{\t}g(\t) + 2\,
\Hess \l_{(p, 0)} + \dertau{g}(\t) \] = 0.
\end{equation*}
Hence,
\begin{equation*}
\frac{1}{\t}(\phi_{\t-\t_\delta})^*g(\t)= \frac{1}{\t_\delta}g(\t_\delta)
\hbox{ or equivalently } g(\t)=
\frac{\t}{\t_\delta}(\phi_{\t-\t_\delta}^{-1})^*g(\t_\delta).
\end{equation*}
Since $g(\t)$ is smooth around $(p, 0)$, we have
\begin{align*}
|\Rm|(q, \t) &= \frac{\t_\delta}{\t}|\Rm|(\phi_{\t-\t_\delta}^{-1}(q), \t_\delta)\\
&\le \frac{\t_\delta}{\t}\(|\Rm|(p, 0) +
\theta(\t_\delta)\) \to 0 \hbox{ as } \t_\delta \to 0
\end{align*}
where $\Rm$ denotes the Riemann curvature tensor and
$\theta(\t_\delta)$ is a function such that
$\theta(\t_\delta)\to 0$ as $\t_\delta \to 0$.
Consequently, $(M^n, g(\t))$ is flat and hence isometric to
$(\mathbb{R}^n, g_{\rm E})$ for each $\t\in[0, \bt]$. We can
write $g(\t) = u(\t)^{-1}g_{\rm E}$ for some positive
non-decreasing function $u(\t)$ with $u(0)=1$. It remains
to show that $u(\t) =1$ for all $\t \in [0, \bt]$.

Introduce a new parameter $\sigma:=2\sqrt\t$ to write
$g(\sigma)=u(\sigma)^{-1}g_{\rm E}$ for $\sigma\in[0, \bs]$, where
$\bs:=2\sqrt{\bt}$. By calculation (cf.~\cite[Lemma 7.67]{V2}),
it is easy to\break see that
\begin{equation*}
\l_{(p, 0)}(q, \bt)= \frac{d_{\rm E}(p, q)^2}{\bs\int_0^{\bs}
u(\sigma)\,d\sigma}  - \frac{n}{2}\log u(\bs) +
\frac{n}{2}\frac{\int_0^{\bs} \log u(\sigma)\,d\sigma}{\bs}
\end{equation*}
and
\begin{align*}
\tilV_{(p, 0)}(\bt) &=
\int_{\R^n} \(\pi{\bs}^2 \exp{ \frac{\int_0^{\bs}\log u(\sigma)\,d\sigma}{\bs} }\)^{-n/2} \exp
-\frac{d_{\rm E}(p, q)^2}{\bs \int_0^{\bs} u(\sigma)\,d\sigma} \,dx(q)\\
&=\( \frac{1}{\bs} \int_0^{\bs} u(\sigma)\,d\sigma
\)^{n/2}\( \exp{ \frac{\int_0^{\bs}\log
u(\sigma)\,d\sigma}{\bs} } \)^{-n/2}.
\end{align*}

It follows from Jensen's inequality
\begin{equation}\label{Jensen}
\frac{1}{\bs} \int_0^{\bs} \log u(\sigma)\,d\sigma \le \log
\frac{1}{\bs} \int_0^{\bs} u(\sigma)\,d\sigma
\end{equation}
that $\tilV_{(p, 0)}(\bt) \ge 1$. As  $\tilV_{(p, 0)}(\bt)
\le 1$, we must have equality in (\ref{Jensen}), that is,
$u(\bs)=1$. This completes the proof of
Theorem~\ref{reduced}.
\end{proof}

\subsection{Example}

As an important example, let us
look at a stationary super Ricci flow. Then we obtain an
invariant which is called the {\it static reduced volume}
in~\cite{V2}. Its relation to the volume ratio is given by
\begin{lem}[{\rm \cite[Lemma 8.10]{V2}}]\label{Ricciflat}
Let $(M^n, g)$ be an $n$-dimensional complete Riemannian
manifold of non-negative Ricci curvature regarded as a
stationary super Ricci flow, i.e., $\dertau{}g=0 \le
2\,\Ric$. Then for any $p \in M$ and $\t>0$, we have
\begin{equation}\label{staticred}
\tilV_{(p, 0)}(\t) = \int_M
(4\pi\t)^{-n/2}\exp\({-\frac{d(p, q)^2}{4\t}}\)\,d\mu(q)
\le 1,
\end{equation}
and
\begin{equation*}
\tilNu(g) :=\lim_{\t \rightarrow\infty} \tilV_{(p, 0)}(\t)
= \lim_{r \rightarrow\infty} \frac{\vol\, B(p, r)}{\omega_n
r^n} =:\nu(g).
\end{equation*}
Furthermore, the equality holds in $(\ref{staticred})$ for
some $\t>0$ if and only if $(M^n, g)$ is isometric to
$(\R^n, g_{\rm E})$.
\end{lem}
By virtue of Lemma~\ref{Ricciflat}, we know that
Theorem~\ref{main} generalizes\break \hbox{Theorem}~\ref{gap}.

One can easily compute how the reduced distance and reduced
volume change under parabolic rescaling.
\begin{prop}[{\rm \cite[Lemma 8.34]{V2}}]\label{rescaling}
If $g(\t), \t\in [0, T)$ is a super Ricci flow, then
$(Qg)(\t) := Qg(Q^{-1}\t), \t\in [0, QT)$ is also a super
Ricci flow for any $Q>0$. Under this parabolic
rescaling, we have
\begin{equation*}
\l^{Qg}(q, \t)=\l^{g}(q, Q^{-1}\t) \hbox{ and }
\tilV^{Qg}(\t) = \tilV^{g}(Q^{-1}\t).
\end{equation*}
In particular, the asymptotic reduced volume is invariant
under the parabolic rescaling, i.e.,
$\tilNu(g)=\tilNu(Qg)$, for any ancient super Ricci flow
$g(\t), \t\in[0, \infty)$.
\end{prop}

\section{Preliminary results}

In this section, we prove some lemmas needed in the proof
of our main theorem.

\subsection{Preliminary estimates}

Given a super Ricci flow $(M^n, g(\t)), \t\in [0, T)$, take
$p\in M$ and $\t\in(0, T)$. Let us put
\begin{equation*}
\L B_\t(p, r) := \{ \L\,\exp_\t(V) ; \ V\in \Omega_{(p,
0)}(\t), |V|_{g(0)}< r \}.
\end{equation*}
This notation comes from the fact that a geodesic ball in a
Riemannian manifold is the image of ball of the same radius
in the tangent space under the exponential map. In this
subsection, we derive a few estimates which we shall make
heavy use of in the remaining of this paper.

\begin{prop}\label{epsilon}
Let $u(\cdot, \t) := (4\pi\t)^{-n/2}\exp(-\l_{(p, 0)}(\cdot, \t))$.
\begin{enumerate}
\item[(1)]For all $r>0$ and $\t\in(0, T)$, we have
\begin{equation*}
\tilV_{(p, 0)}(\t) -\epsilon(r) \le \int_{\L B_\t(p, r)}
u(\cdot, \t)\, d\mu_{g(\t)}.
\end{equation*}
\item[(2)]Given $r>0$ and $\t_0\in (0, T)$,
we can find a family of subsets $\L K_{\t, \t_0}(p, r)$ of $M$ for $\t\in(0, T)$ satisfying the following properties:
\begin{enumerate}
\item[(a)]For all $\t \le \t_0,\L K_{\t, \t_0}(p, r)$ is compact.
\item[(b)]For all $\t \le \bt,\L K_{\t, \t_0}(p, r)$ contains all of the points $\gamma(\t)$
on any minimal $\L$-geodesics $\gamma :[0, \bt]\to M$ connecting $(p, 0)$ and $(q, \bt)$
with $q \in \L K_{\bt, \t_0}(p, r)$.
\item[(c)]For all $\t\ge\t_0$ we have
\begin{equation*}
\tilV_{(p, 0)}(\t) - 2\epsilon(r) \le \int_{\L K_{\t,
\t_0}(p, r)}  u(\cdot, \t)\,d\mu_{g(\t)}.
\end{equation*}
\end{enumerate}
\end{enumerate}

Here, $\epsilon(r)$ is a function of $r>0$ with
$\epsilon(r) \le {\rm e}^{-r^2/2}$ for all $r$ large
enough. Clearly, $\epsilon(r)$ decays to $0$ exponentially
as $r\to \infty$.
\end{prop}

\begin{proof}
(1) We deduce from~(\ref{eq:reducedpullback}) and
Theorem~\ref{reduced} that
\begin{align*}
\int_{M \setminus \L B_\t(p, r)} u(\cdot, \t)\, d\mu_{g(\t)}&=\int_{ \Omega_{(p, 0)}(\t) \setminus B(0, r) } u(\L\,\exp_\t(V), \t) \L J_V(\t) dx_{g(0)}(V)\\
&\le\int_{T_{p}M \setminus B(0, r)}
\pi^{-n/2}e^{-|V|_{g(0)}^2}\, dx_{g(0)}(V)=:\epsilon(r).
\end{align*}

(2) Take a compact set $K$ of $T_pM$ so that $K \subset
B(0, r)\cap \Omega_{(p, 0)}(\t_0)$ and the Lebesgue measure
of $B(0, r)\cap \Omega_{(p, 0)}(\t_0)\setminus K$, induced
by $g(0)$, is less than $\pi^{n/2}\epsilon(r)$. We show
that $\L K_{\t, \t_0}(p, r) :=
\L\,\exp_{\t}(K\cap\Omega_{(p, 0)}(\t))$ has the desired
properties. It is clear that (a) and (b) hold by
construction, since $\Omega_{(p, 0)}(\t_0) \subset
\Omega_{(p, 0)}(\t)$ for $\t\le \t_0$. Furthermore, by the
same argument as in (1), we deduce that
\begin{align*}
&\int_{M \setminus \L K_{\t, \t_0}(p, r)} u(\cdot, \t)\, d\mu_{g(\t)}\\
&\quad=\int_{M\setminus \L B_\t(p, r)} +\int_{\L B_\t(p, r)
\setminus \L K_{\t, \t_0}(p, r)} u(\cdot, \t)\, d\mu_{g(\t)}
\quad \le 2\epsilon(r)
\end{align*}
for $\t\ge \t_0$.

Finally, we estimate $\epsilon(r)$ for $r\ge r_0$ by
\begin{equation*}
\epsilon(r) = n\omega_n \pi^{-n/2} \int_r^\infty
{\rm e}^{-r^2}r^{n-1}\,dr \le \int_r^\infty {\rm e}^{-r^2/2}r\,dr
= {\rm e}^{-r^2/2}.
\end{equation*}
Here $r_0\gg1$ is taken so that
$n\omega_n\pi^{-n/2}{\rm e}^{-r^2/2}r^{n-2} \le 1$ for all $r \ge
r_0$.
\end{proof}

\begin{prop}\label{escape}
Assume that $h\kern-1pt  \ge\kern-1pt  -C_0g(\t)$ and $|\nabla H|^2\kern-1pt  \le\kern-1pt  D_0$ on
$\mathcal{K}\kern-1pt \times\kern-1pt [0, T_0]$ for some compact set
$\mathcal{K} \subset M$ containing a ball $B_{g(0)}(p, r)$.
Consider the \hbox{$\L$-geodesic} $\gamma_V :[0, \bt] \to M$ with
$\gamma_V(0)=p$ and $\lim_{\t \to
0+}\sqrt\t\frac{d\gamma_V}{d\t}=V$. Then we can find
$C=C(C_0, T_0), D=D(C_0, D_0, T_0)$ and small $\delta =
\delta(C_0, D_0, r,\break  |V|_{g(0)}, T_0)>0$ such that
\begin{equation}\label{ineq:escape}
d_{g(0)}(p, \gamma_V(\t)) \le (C|V|_{g(0)}+D)\sqrt\t
\end{equation}
and hence $\gamma_V(\t) \in B_{g(0)}(p, r) \subset
\mathcal{K}$ for all $\t \in[0, \delta]$.
\end{prop}

\begin{proof}
Let $\t^\prime \in[0, T_0]$ be the maximal time such that
$\gamma_V([0, \t^\prime]) \subset \mathcal{K}$. For $\t \le
\t^\prime$, we use the $\L$-geodesic equation~(\ref{geod})
to obtain
\begin{align*}
\frac{d}{d\t} |\sqrt\t X|^2_{g(\t)}
&= |X|_{g(\t)}^2 + 2h(\sqrt\t X, \sqrt\t X) + 2\t\langle\nabla_X X, X\rangle\\
&= -2h(\sqrt\t X, \sqrt\t X) + \t\langle\nabla H, X\rangle\\
&\le -2h(\sqrt\t X, \sqrt\t X) + |\sqrt\t X|^2_{g(\t)} + \t |\nabla H|^2_{g(\t)}\\
&\le (2C_0+1)|\sqrt\t X|^2_{g(\t)} + D_0T_0.
\end{align*}

From this, we derive that
\begin{align*}
|\sqrt\t X|^2_{g(\t)}
&\le {\rm e}^{(2C_0+1)\t}|V|_{g(0)}^2 + D_0T_0({\rm e}^{(2C_0+1)\t}-1)\\
&\le (C|V|_{g(0)}+D)^2
\end{align*}
for $C=C(C_0, T_0)$ and $D=D(C_0, D_0, T_0)$, and hence
\begin{align*}
d_{g(0)}(p, \gamma_V(\t))
&\le \int_0^\t |X|_{g(0)}\,d\t \le \int_0^\t {\rm e}^{C_0\t}|X|_{g(\t)}\,d\t\\
&\le (C|V|_{g(0)}+D)\int_0^\t \t^{-1/2}\,d\t=
(C|V|_{g(0)}+D)\sqrt\t.
\end{align*}
As a consequence, we can find $\delta = \delta(C, D, r,
|V|_{g(0)}, T_0) >0$ such that (\ref{ineq:escape}) holds
for $\t\in[0, \delta]$.
\end{proof}

\subsection{Asymptotic reduced volume}

Given an ancient super Ricci flow $(M, g(\t)), \t\in[0,
\infty)$, it is natural to expect that the asymptotic
reduced volume $\tilNu(g):= \lim_{\t\to\infty}
\tilV^{g}_{(p, 0)}(\t)$ is well defined, namely it does not
depend on $p\in M$, as the asymptotic volume ratio is. In
this subsection, we prove the following
\begin{lem}\label{basept}
Let $(M^n, g(\t)), \t \in[0, \infty)$ be a complete ancient
super Ricci flow satisfying Assumption~$\ref{as}$ with time
derivative bounded from below. Then for any $(p_k, \t_k)
\in M \times [0, \infty)$ for $k=1, 2$ with $\t_2\ge \t_1$,
we have
\begin{equation*}
\lim_{\t \to \infty} \tilV^{g_2}_{(p_2, 0)}(\t) \ge
\lim_{\t \to \infty} \tilV^{g_1}_{(p_1, 0)}(\t)
\end{equation*}
where $g_k(\t):=g(\t+\t_k), \t\in[0, \infty)$. In
particular, $\tilNu(g)$ is well defined.
\end{lem}

\makeatletter

\renewcommand{\openbox}{}

\makeatother

\begin{proof}
Put $\t_\Delta :=\t_2-\t_1 \ge0$ to notice that
$g_2(\t-\t_\Delta) = g_1(\t)$. We first\break \hbox{verify}
\end{proof}

\makeatletter

\renewcommand{\openbox}{\leavevmode
  \hbox to .77778em{\hfil${\square}$\hfil}}

\makeatother

\begin{sublem}\label{timeshift}
For any $(p, \t_p), (q, \bt)\in M\times [0, \infty)$ with $\bt > \t_p \ge \t_\Delta$,
\begin{equation*}
\frac{1}{2\sqrt{\bt-\t_\Delta}} L^{g_2}_{(p,
\t_p-\t_\Delta)}(q, \bt-\t_\Delta)\le \frac{1}{2\sqrt{\bt}}
L^{g_1}_{(p, \t_p)}(q, \bt)
\end{equation*}
and
\begin{equation*}
\frac{1}{2\sqrt{\bt-\t_\Delta}} L^{g_2}_{(p,
\t_p-\t_\Delta)}(q, \bt-\t_\Delta) \ge
\alpha^{\t_\Delta}_{\t_p, \bt} \frac{1}{2\sqrt{\bt}}
L^{g_1}_{(p, \t_p)}(q, \bt),
\end{equation*}
where $\alpha^{\t_\Delta}_{\t_p, \bt}  := \sqrt{\frac{\t_p-\t_\Delta}{\t_p}\frac{\bt}{\bt-\t_\Delta}} \ge
\sqrt{1-\frac{\t_\Delta}{\t_p}}$.
\end{sublem}

\begin{proof}
We use the fact that $H(\cdot, \t) \ge0$ for ancient super Ricci
flows\break (Proposition~A.1) and the inequality
\begin{equation*}
\frac{1}{2\sqrt{\bt}}\sqrt{\t} \ge
\frac{1}{2\sqrt{\bt-\t_\Delta}}\sqrt{\t-\t_\Delta}\quad
\hbox{ for all } \t_p \le \t \le \bt
\end{equation*}
to obtain
\begin{align*}
\frac{1}{2\sqrt{\bt}} L^{g_1}_{(p, \t_p)}(q, \bt)
&= \frac{1}{2\sqrt{\bt}} \inf_{\gamma} \biggl\{ \int_{\t_p}^{\bt} \sqrt{\t}\( |\gamma^\prime|_{g_1(\t)}^2+H_{g_1(\t)}(\gamma(\t))\) d\t\biggr\}\\
&\ge \frac{1}{2\sqrt{\bt-\t_\Delta}} \inf_{\gamma} \biggl\{ \int_{\t_p}^{\bt} \sqrt{\t-\t_\Delta}\(|\gamma^\prime|_{g_1(\t)}^2+H_{g_1(\t)}(\gamma(\t))\) d\t\!\biggr\}\!\\
&= \frac{1}{2\sqrt{\bt-\t_\Delta}} L^{g_2}_{(p, \t_p-\t_\Delta)}(q, \bt-\t_\Delta).
\end{align*}
Here $\inf$ runs over all curves $\gamma :[\t_p, \bt]\to M$
with $\gamma(\t_p)=p$ and $\gamma(\bt)=q$.

To see the second inequality,
we use instead
\begin{equation*}
\alpha^{\t_\Delta}_{\t_p,
\bt}\frac{1}{2\sqrt{\bt}}\sqrt{\t} \le
\frac{1}{2\sqrt{\bt-\t_\Delta}}\sqrt{\t-\t_\Delta} \hbox{
for all } \t_p \le \t \le \bt.
\end{equation*}
\end{proof}

We return to the proof of the lemma. Fix $r>0$ and
$\bt\gg1$. Take $q\in \K(\bt) := \L^{g_1} K_{\bt,
2\t_\Delta}(p_1, r)$ and the point $p_\Delta
=\gamma(2\t_\Delta) \in M$ on the minimal
$\L^{g_1}$-geodesic $\gamma :[0, \bt]\to M$ from $(p_1, 0)$
to $(q, \bt)$ such that
\begin{align}\label{eq:midpt}
\begin{split}
L^{g_1}_{(p_1, 0)}(q, \bt)
&= L^{g_1}_{(p_\Delta, 2\t_\Delta)}(q, \bt) + L^{g_1}_{(p_1, 0)}(p_\Delta, 2\t_\Delta)\\
&\ge L^{g_1}_{(p_\Delta, 2\t_\Delta)}(q, \bt).
\end{split}
\end{align}
The inequality in~(\ref{eq:midpt}) is due to the
non-negativity of $H$. Recall that $\K := \L^{g_1}
K_{2\t_\Delta, 2\t_\Delta}(p_1, r)$ is compact and
$p_\Delta \in \K$ by construction
(Proposition~\ref{epsilon}). It follows from the
combination of the triangle inequality for $\L$-distance,
Sublemma~\ref{timeshift} and (\ref{eq:midpt}) that
\begin{align*}
\l^{g_2}_{(p_2, 0)}(q, \bt-\t_\Delta)
&\le \frac{1}{2\sqrt{\bt-\t_\Delta}} \(  L^{g_2}_{(p_\Delta, \t_\Delta)} (q, \bt-\t_\Delta) + L^{g_2}_{(p_2, 0)} (p_\Delta, \t_\Delta) \)\\
&\le \frac{1}{2\sqrt{\bt}} L^{g_1}_{(p_\Delta, 2\t_\Delta)} (q, \bt) + \frac{1}{2\sqrt{\bt-\t_\Delta}} \max_{\K} L^{g_2}_{(p_2, 0)} (\cdot, \t_\Delta)\\
&\le \l^{g_1}_{(p_1, 0)}(q, \bt) + C(r)\bt^{-1/2}.
\end{align*}

Thus, as $\bt>0$ is large enough,
\begin{align*}
\lim_{\t\to\infty}\tilV_{(p_2, 0)}^{g_2}(\t)
&\ge \tilV_{(p_2, 0)}^{g_2}(\bt-\t_\Delta) -\epsilon(r)\\
&\ge \int_{\K(\bt)} (4\pi\bt)^{-n/2}\exp\({-\l^{g_2}_{(p_2, 0)}(\cdot, \bt-\t_\Delta)}\) d\mu_{g_2(\bt-\t_\Delta)} -\epsilon(r)\\
&\ge {\rm e}^{-C(r)\bt^{-1/2}} \int_{\K(\bt)} (4\pi\bt)^{-n/2}\exp\({-\l^{g_1}_{(p_1, 0)}(\cdot, \bt)}\) d\mu_{g_1(\bt)} -\epsilon(r)\\
&\ge {\rm e}^{-C(r)\bt^{-1/2}} \tilV^{g_1}_{(p_1, 0)}(\bt) -3\epsilon(r)\\
&\ge {\rm e}^{-C(r)\bt^{-1/2}} \lim_{\t\to\infty}\tilV_{(p_1, 0)}^{g_1}(\t) -3\epsilon(r).
\end{align*}
We have used Proposition~\ref{epsilon} to derive the fourth
inequality. Since $\bt>0$ and $r>0$ are arbitrary, the
proof of Lemma~\ref{basept} is now complete.\hfill$\square$

\subsection{Finiteness of fundamental group}

Now we are ready to establish
Theorem~\ref{finitefundaRicci}. As mentioned in the
introduction, what we intend to show is the following.
\begin{lem}\label{funda}
Let $(M^n, g(\t)), \t \in[0, \infty)$ be a complete ancient
super Ricci flow satisfying Assumption~$\ref{as}$ with time
derivative bounded below. We lift them to the universal
covering $\bM$ of $M$ to obtain the lifted flow $(\bM,
\bg(\t))$. Take $p \in M$ and $\bp \in \pi^{-1}(p)$, where
$\pi : \bM \to M$ is the projection. Suppose that
$\tilNu(g) := \lim_{\t \to \infty}\tilV^g_{(p, 0)}(\t) >0$.
Then we have
\begin{equation*}
|\pi_1(M)| = \tilNu(\bar{g}) \tilNu(g)^{-1} < +\infty.
\end{equation*}
\end{lem}

Before we begin the proof of Lemma~\ref{funda}, let us state the
following immediate corollary, which follows from
Lemma~\ref{funda} combined with Lemma~\ref{Ricciflat}.

\begin{cor}[{\rm \cite{A2,Li}}]\label{AL}
Let $(M, g)$ be a complete Riemmanian manifold with non-negative Ricci curvature
and $(\bM, \bar{g})$ be the universal covering of $(M, g)$.
If $(M, g)$ has Euclidean volume growth, i.e., $\nu(g) >0$, then we have
\begin{equation*}
|\pi_1(M)| =\nu(\bar{g})\nu(g)^{-1} < +\infty.
\end{equation*}
Here, $\nu(g)$ denotes the asymptotic volume ratio as
before.
\end{cor}

\begin{proof}[Proof of Lemma~$\ref{funda}$]
The proof is a modification of that of \cite[Theorem 1.1]{A2}.
Fix large $\bt \in(0, \infty)$ and define
\begin{equation*}
F := \bigcap_{\alpha\in \pi_1(M) \setminus \{e\}} \Bigl\{ \bq \in \bM ;\  L^{\bar{g}}_{(\bp, 0)}(\bq, \bt) < L^{\bar{g}}_{(\alpha\bp, 0)}(\bq, \bt) \Bigr\}.
\end{equation*}
Then $F$ is a fundamental domain of $\pi:\bM\to M$, namely
\begin{equation*}
F \cap \alpha F = \emptyset\quad \hbox{ for } \alpha \in
\pi_1(M)\setminus\{e\}\quad \hbox{and}\enspace
\bigcup_{\alpha\in \pi_1(M)} \alpha\bar{F} = \bM.
\end{equation*}

We claim that $\pi : \bar{F} \to M$ is locally isometric
and surjective. To see this, pick $q \in M$ and connect
$(p, 0)$ and $(q, \bt)$ by a minimal $\L^{g}$-geodesic
$\gamma :[0, \bt] \to M$. Then the lift $\bar{\gamma}$ of
$\gamma$ with $\bar{\gamma}(0)=\bp$ is a minimal
$\L^{\bar{g}}$-geodesic in $\bM$. Let $\bq :=
\bar{\gamma}(\bt)$. Then we have that $\bq \in \bar{F}$
and~$\pi(\bq)=q$.

Furthermore, $\bar{F}\setminus F$ has measure 0, since
$\pi(\bar{F}\setminus F)$ consists of the points in $M$
such that minimal $\L^{g}$-geodesic from $(p, 0)$ is not
unique. The set of such points has measure 0~\cite[Lemma
7.99]{V2}.

Fix any finite subset $\Gamma \subset \pi_1(M)$ and set
$D_\Gamma := \max\{ d_{\bg(0)}(\bp, \alpha\bp) ; \alpha \in
\Gamma \}$. Take $C_0<\infty$ such that $|h|\le C_0$ on
$B_{\bg(0)}(\bp, D_\Gamma+1)\times [0, 1]$ and $|\nabla
H|^2 \le C_0$ on $B_{g(0)}(p, 1)\times [0, 1]$. Fix $r>0$.
Due to Proposition~\ref{escape}, we can find $\delta
=\delta(C_0, r) >0$ such that
$d_{\bg(0)}(\bar{\gamma}_V(\t), \alpha\bp) \le 1$ for any
$\L^{\bg}$-geodesic $\bar{\gamma}_V$ starting from
$\alpha\bp$ with $|V|_{\bg(0)} < r$ and $\t \in
[0,~\delta]$.

For any $\alpha\kern-1pt \in\kern-1pt  \Gamma$ and $\bq \in \L B_{\bt}(\alpha
\bp, r) \cap \alpha\bar{F}$, let $\bar{\gamma}$ be the
minimal $\L^{\bg}$-geodesic from $(\alpha\bp, 0)$ to $(\bq,
\bt)$ in $\bM$ and connect $\bp$ and $\bar{\gamma}(\delta)$
by a minimal $\bg(0)$-geodesic $\xi_{\bp,
\bar{\gamma}(\delta)} :[0, \delta]\to \bM$. Define a curve
$\hat{\gamma}: [0, \bt]\to \bM$ by
\begin{equation*}
\hat{\gamma}(\t) :=
\begin{cases}
\xi_{\bp, \bar\gamma(\delta)}(\t) &\hbox{ on } [0, \delta]\\
\bar{\gamma}(\t) &\hbox{ on } [\delta, \bt]
\end{cases}
\end{equation*}
Then, letting $q:=\pi(\bq)$,
\begin{align*}
\l^{\bg}_{(\bp, 0)}(\bq, \bt)
&\le \frac{1}{2\sqrt{\bt}} \L^{\bg}(\hat{\gamma})\\
&= \frac{1}{2\sqrt{\bt}} \( \L^{\bg}(\bar\gamma) - \L^{\bg}(\bar\gamma|_{[0, \delta]}) + \L^{\bg}(\xi_{\bp, \bar\gamma(\delta)}) \)\\
&\le \l^{\bg}_{(\alpha\bp, 0)}(\bq, \bt) + \frac{1}{3\sqrt{\bt}} {\delta}^{3/2} \( {\rm e}^{2C_0\delta}\(\frac{D_\Gamma+1}{\delta}\)^2+2nC_0 \)\\
&= \l^{g}_{(p, 0)}(q, \bt) + C(\delta, \Gamma) \bt^{-1/2}
\end{align*}
where we have used that
\begin{equation*}
\l^{\bg}_{(\alpha\bp, 0)}(\bq, \bt) = \l^{g}_{(p, 0)}(q,
\bt)\quad \hbox{for any } \bq \in \L B_{\bt}(\alpha\bp, r)
\cap \alpha\bar{F}.
\end{equation*}

We apply Proposition~\ref{epsilon} to obtain that
\begin{align*}
\tilV^{\bg}_{(\bp, 0)}(\bt)
&\ge \sum_{\alpha \in \Gamma} \int_{ \L B_\t(\alpha\bp, r) \cap \alpha\bar{F} } (4\pi\bt)^{-n/2}\exp\(-\l^{\bar{g}}_{(\bp, 0)}(\cdot, \bt)\) d\mu_{\bar{g}(\bt)}\\
&\ge |\Gamma| \int_{\L B_\t(p, r)} (4\pi\bt)^{-n/2}\exp\(-\l^g_{(p, 0)}(\cdot, \bt)- C(\delta, \Gamma) \bt^{-1/2}\) d\mu_{g(\bt)}\\
&\ge {\rm e}^{-C(\delta, \Gamma) \bt^{-1/2}}|\Gamma|
\(\tilV^{g}_{(p, 0)}(\bt) -\epsilon(r) \)
\end{align*}
and taking $\bt \to \infty$ and $r \to \infty$ yields that
\begin{equation}\label{finitefunda}
\tilNu(\bg) \ge |\Gamma| \tilNu(g) \hbox{ for any finite
subset } \Gamma \subset \pi_1(M).
\end{equation}
Thus, $\pi_1(M)$ is finite and~(\ref{finitefunda}) holds
for $\Gamma = \pi_1(M)$.

On the other hand, since
\begin{equation*}
\l^{\bg}_{(\bp, 0)}(\bq, \t) \ge \l^{g}_{(p, 0)}(\pi(\bq),
\t) \hbox{ for any } (\bq, \t) \in \bM \times (0,
\infty)
\end{equation*}
we have
\begin{equation*}
\tilV^{\bg}_{(\bp, 0)}(\t) \le |\pi_1(M)| \tilV^{g}_{(p, 0)}(\t)
\end{equation*}
and hence $\tilNu(\bg) \le |\pi_1(M)| \tilNu(g)$. This
finishes the proof of the lemma.
\end{proof}

We close this subsection by giving another corollary of
Lemma~\ref{funda}.
\begin{cor}\label{kappafunda}
Any ancient $\kappa$-solution to the Ricci flow has finite fundamental group.
\end{cor}

The proof is immediate since any ancient $\kappa$-solution has
positive\break asymptotic reduced volume~\cite[Lemma 8.38]{V2}.
Meanwhile, Perelman has\break shown that {\it any ancient
$\kappa$-solution has zero asymptotic volume ratio}
$\nu(g(\t))$\break \cite[Proposition 11.4]{P} (cf. \cite{CN}). This is
why Corollary~\ref{kappafunda} does not follow from
Corollary~\ref{AL}, but from Lemma~\ref{funda}.
See~\cite[Definition 8.31]{V2} for the definition of ancient
$\kappa$-solution.

\subsection{Reduced volume under Cheeger--Gromov convergence}

Although we have considered the super Ricci flow so far,
Theorem \ref{main} is not true for them. From now on, we
concentrate on the Ricci flow. To begin with, let us recall
Shi's gradient estimate. Shi's derivative estimate was also
employed in the proof of the compactness theorem for the
Ricci flow~\cite{Ha-Comp}, which we will use later.

\begin{thm}[{\rm (Shi's local gradient estimate {\cite[Theorem 13.1]{Ha-Form}})}]\label{Shi}
There exists a constant $C(n)\kern-1pt <\kern-1pt\infty$ satisfying the following:
let $(M^n, g(\t)), \t\in[0, T_0]$ be a complete backward Ricci flow on an $n$-manifold $M$. Assume that the ball $B_{g(T_0)}(p, r)$ is contained in $\mathcal{K}$ and $|\Rm| \le C_0$ on $\mathcal{K}\times [0, T_0]$ for some compact set $\mathcal{K} \subset M$.
Then for $\t \in[0, T_0)$,
\begin{equation}
|\nabla \Rm|^2(p, \t) \le C(n)C_0^2 \(\frac{1}{r^2} +
\frac{1}{T_0-\t} + \frac{1}{C_0} \).
\end{equation}
\end{thm}

Recall that we say that a sequence of pointed backward
Ricci flows
\begin{equation*}
\{(M_k^n, g_k(\t), p_k)\}_{k\in \Z}, \t\in [0, T)
\end{equation*}
converges to a backward Ricci flow $(M_\infty^n,
g_\infty(\t), p_\infty), \t\in[0, T)$ in the $C^\infty$
Cheeger--Gromov sense if there exist open sets $U_k$ of
$M_\infty$ with $p_\infty \in U_k$ and $\cup_{k\in \Z}U_k =
M_\infty$ and deffeomorphisms $\Phi_k : U_k \to V_k:=
\Phi_k(U_k) \subset M_k$ with $\Phi_k(p_\infty)=p_k$ so
that $\{(U_k, \Phi_k^*g_k(\t))\}_{k\in \Z}$ converges to
$(M_\infty^n, g_\infty(\t))$ in the $C^\infty$ topology on
each compact set of $M_\infty^n \times [0, T)$.

By carefully investigating the proof of~\cite[Lemma
7.66]{V2}, where curvature is assumed to be bounded on the
whole of $M_k\times [0, T)$, one can show the following
lemma without modification (cf. [8, Lemma 7.66]).
\begin{lem}\label{limdis}
Let $\{(M_k^n, g_k(\t), p_k)\}_{k\in\Z}, \t\in [0, T)$ be a
converging sequence of pointed backward Ricci flows in the
sense of $C^\infty$ Cheeger--Gromov and $(M_\infty^n,
g_\infty(\t), p_\infty), \t\in[0, T)$ be the limit. Then we
have
\begin{equation}\label{limsupdis}
\limsup_{k \to \infty} \l_{(p_k, 0)}^{g_k}(\Phi_k(q), \t) \le \l_{(p_\infty, 0)}^{g_\infty}(q, \t)
\end{equation}
for $\t\in(0, T)$. The equality is achieved in
$(\ref{limsupdis})$, with $\limsup$ replaced by $\lim$,
provided $(\Phi_k(q), \t)$ can be joined to $(p_k, 0)$ by a
minimal $\L^{g_k}$-geodesic within the image
$\Phi_k(\mathcal{K}) \subset M_k$ of some compact set
$\mathcal{K} \subset M_\infty$ for all large $k\in\Z$.
\end{lem}

Now we verify the convergence of reduced volumes.
\begin{lem}\label{limvol}
Let $\{(M^n_k, g_k(\t), p_k)\}_{k\in \Z}, \t\in [0, T)$ be
a sequence of pointed backward Ricci flows converging to
$(M^n_\infty, g_\infty(\t), p_\infty)$. Assume that
\begin{equation*}
|\Rm| \le C_0 \hbox{ on } V_k \times [0, T) \hbox{ and }
\bigcup_{k\in \Z}U_k =M_\infty.
\end{equation*}
Then for any $\t \in (0, T)$,
\begin{equation}
\lim_{k \to \infty} \tilV_{(p_k, 0)}^{g_k}(\t) = \tilV_{(p_\infty, 0)}^{g_\infty}(\t).
\end{equation}
\end{lem}

\begin{proof}
Let us put $u_{\star} (q, \t) := (4\pi\t)^{-n/2}
\exp(-\l^{g_\star}_{(p_{\star}, 0)}(q, \t))$ for $\star\in
\Z\cup\{\infty\}$, and fix $\bt \in(0, T)$ and $T_0 \in
(\bt, T)$. Set $V_\infty := M_\infty$.

We invoke Shi's gradient estimate (Theorem~\ref{Shi}):
\begin{equation*}
|\nabla R|^2(\cdot, \t) \le \frac{C(n)C_0^2}{\min\{C_0,
T_0-\t\}} \hbox{ on } B_{[0, T_0]}(V_\star, -\sqrt{C_0})
\hbox{ for }\t \in [0, T_0)
\end{equation*}
where
\begin{equation*}
B_{[0, T_0]}(V_\star, -\sqrt{C_0}) := \{ x\in V_\star ;
B_{g_\star(\t)}(x, \sqrt{C_0}) \subset V_\star \hbox{ for
all } \t \in [0, T_0]\}.
\end{equation*}
Fix $r>0$. Then by Proposition~\ref{escape}, we can find
$C(r)<\infty$ such that any\break $\L^{g_\star}$-geodesic
$\gamma_V([0, \bt])$ in $M_\star$ with $\gamma_V(0)=p_\star$ and
$|V|_{g_\star(0)}\le r$ can not escape from $B_{0}(p_\star,
C(r))$ when $\star$ is sufficiently large or~$=\infty$.

Define $\hat{u}_k(\cdot, \t) : M_k \to [0, \infty)$ by
\begin{equation*}
\hat{u}_k(q, \t) :=
\begin{cases}
u_k(q, \t) &\hbox{if } q \in \L B_\t(p_k, r)\\
0 &\hbox{otherwise}
\end{cases}
\end{equation*}
Then each $\hat{u}_k(\cdot, \bt)$ has a compact support in
$B_{g_k(0)}(p_k, C(r))$ and it follows from
Lemma~\ref{limdis} that
\begin{equation}\label{limsupu}
\limsup_{k\to\infty}\ \hat{u}_k(\Phi_k(q), \bt) \in \{ u_\infty(q, \bt), 0\}.
\end{equation}

Therefore, noting that $\hat{u}_k(\cdot, \bt)\kern-1pt \le\kern-1pt
(4\pi\bt)^{-n/2}\exp(\frac{1}{3}n(n-1)C_0\bt)$, we derive
from Proposition~\ref{epsilon}, Fatou's lemma and
(\ref{limsupu}) that
\begin{align*}
\limsup_{k \to\infty} \tilV^{g_k}_{(p_k, 0)}(\bt) -\epsilon(r)
&\le \limsup_{k\to\infty} \int_{ \L B_{\bt}(p_k, r) } u_k(\cdot, \bt)\,d\mu_{g_k(\bt)}\\
&= \limsup_{k\to\infty} \int_{B_{0}(p_\infty, C(r))} \hat{u}_k(\Phi_k(\cdot), \bt)\,d\mu_{\Phi_k^*g_k(\bt)}\\
&\le \int_{B_{0}(p_\infty, C(r))} \limsup_{k\to\infty}\ \hat{u}_k(\Phi_k(\cdot), \bt)\,d\mu_{\Phi_k^*g_k(\bt)}\\
&\le \tilV^{g_\infty}_{(p_\infty, 0)}(\bt).
\end{align*}

On the other hand, by combining Fatou's lemma and
(\ref{limsupdis}), we obtain
\begin{align*}
\liminf_{k \to \infty} \tilV^{g_k}_{(p_k, 0)}(\bt)
&\ge \liminf_{k \to \infty} \int_{\L B_{\bt}(p_\infty, r)} u_k(\Phi_k(\cdot), \bt)\,d\mu_{\Phi_k^{*}g_k(\bt)}\\
&\ge \int_{\L B_{\bt}(p_\infty, r)} \liminf_{k \to \infty} u_k(\Phi_k(\cdot), \bt)\,d\mu_{\Phi_k^{*}g_k(\bt)}\\
&\ge \int_{\, \L B_{\bt}(p_\infty, r)} u_\infty(\cdot, \bt)\,d\mu_{g_\infty(\bt)}\\
&\ge \tilV^{g_\infty}_{(p_\infty, 0)}(\bt) -\epsilon(r).
\end{align*}
We also used Proposition~\ref{epsilon} to get the last
inequality. Since $r>0$ and $\bt\in (0, T)$ are chosen
arbitrarily, we conclude that
\begin{equation*}
\lim_{k \to \infty} \tilV_{(p_k, 0)}^{g_k}(\t) =
\tilV_{(p_\infty, 0)}^{g_\infty}(\t)
\end{equation*}
for any $\t \in (0, T)$. This completes the proof of
Lemma~\ref{limvol}.
\end{proof}

\section{Proof of the main theorem}

Before proceeding to the proof of Theorem~\ref{main}, we
first establish the following technical lemma.
\begin{lem}\label{main2}
For any $\alpha>0$ and $\bt>0$ with $\alpha \bt^{-1} >2$,
we can find $\epsilon_n(\alpha\bt^{-1})\break >0$ depending
on $\alpha\bt^{-1}$ and $n \ge 2$ which satisfies the
following: let $(M^n,\break g(\t)), \t \in [0, T), T <
\infty$ be a complete backward Ricci flow with Ricci
curvature bounded bellow. Put
\begin{equation*}
M(\alpha) := \bigl\{ (p, s) \in M\times [0, T) ;\ |\Rm|(p,
s)>\alpha(T-s)^{-1}\bigr\}.
\end{equation*}
Suppose that the reduced volume based at $(p, s)$ satisfies
\begin{equation*}
\tilV_{(p, s)}(Q_{(p, s)}^{-1}\bt) >
1-\epsilon_n(\alpha\bt^{-1})\quad\hbox{at all } (p, s) \in
M(\alpha)
\end{equation*}
with $Q_{(p, s)} := |\Rm|(p, s)$. Here we define
$\tilV_{(p, s)}(\bt)$ as $\tilV_{(p, 0)}^{g_s}(\bt)$ for
$g_s(\t) := g(\t+s), \t\in[0, T-s)$. Then $M(\alpha) =
\emptyset$, that is,
\begin{equation*}
|\Rm|(\cdot, \t) \le \alpha(T-\t)^{-1} \hbox{ on } M\times
[0, T).
\end{equation*}
\end{lem}

One might notice the similarity of the statement of
Lemma~\ref{main2} to those of Perelman's pseudolocality
theorem~\cite[Theorem 10.1]{P} and Ni's
$\epsilon$-regularity theorem~\cite[Theorem 4.4]{Ni-Mean}.
In fact, the proof of Lemma~\ref{main2} follows the same
line as those of them. (As the referee report says, there
is a close relation between gap and local
regularity~theorems.)

\begin{proof}[Proof of Lemma~$\ref{main2}$]
We prove by contradiction. Fix $\alpha >0$ and $\bt>0$ with
$\alpha \bt > 2$. Assume that we have a sequence $\{(M_k^n,
g_k(\t))\}_{k\in \Z}, \t \in [0, T_k)$ of complete backward
Ricci flows with Ricci curvature bounded below such that
\begin{itemize}
\item $M_k(\alpha) := \{ (p, \t) \in M_k\times [0, T_k) ;\  |\Rm|(p, \t)(T_k-\t) > \alpha \} \ne \emptyset$ and
\item $\tilV^{g_k}_{(p, \t)}(Q_{(p, \t)}^{-1}\bt) > 1-k^{-1}$ for any $(p, \t) \in M_k(\alpha)$, where $Q_{(p, \t)}:=\break |\Rm|(p, \t)$.
\end{itemize}

Applying Perelman's point picking lemma (Lemma
\ref{pointpick}) for $(A, B) = (k, \alpha)$, we can find a
point $(p_k, \t_k) \in M_k(\alpha)$ such that
$\tilV^{g_k}_{(p_k, \t_k)}(Q_k^{-1}\bt) > 1- k^{-1}$ and
\begin{equation*}
|\Rm|(x, \t) \le 2Q_k
\end{equation*}
for $(x, \t) \kern-1pt\in\kern-1pt B_{g_k(\t_k)}(p_k, kQ_k^{-1/2})\kern-1pt \times\kern-1pt
[\t_k, \t_k + \frac{1}{2} Q_k^{-1}\alpha]$, where $Q_k :=
|\Rm|(p_k, \t_k)$.

Consider the sequence $\{(M_k^n, \tilde{g}_k(\t),
p_k)\}_{k\in\Z}$ of rescaled Ricci flows
\begin{equation*}
\tilde{g}_k(\t) := Q_kg_k(Q_k^{-1}\t + \t_k),\  \t\in [0,
\alpha/2].
\end{equation*}
Then every $\tilde{g}_k(\t)$ has $|\Rm|(p_k, 0)=1$,
$|\Rm|\le 2$ on $B_{\tilde{g}_k(0)}(p_k, k)\times [0,
\alpha/2]$, and $\tilV^{\tilde{g}_k}_{(p_k, 0)}(\bt) >
1-k^{-1}$ by Proposition~\ref{rescaling}.

Now we observe that the injectivity radius of $(M_k,
\tilde{g_k}(0))$ at $p_k$ is uniformly bounded from below.
To see this, we use Proposition~\ref{escape} to get small
$\delta = \delta(r)>0$ so that $\L\,\exp_\delta(p_k, r)
\subset B_{\tilde{g}_k(0)}(p_k, 1)$ for some large $r>0$
and all large $k$. Then
\begin{align*}
1- k^{-1} < \tilV^{\tilde{g}_k}_{(p_k, 0)}(\bt) &\le \tilV^{\tilde{g}_k}_{(p_k, 0)}(\delta)\\
&\le
(4\pi\delta)^{-n/2}e^{n(n-1)\delta}\vol_{\tilde{g}_k(0)}B_{\tilde{g}_k(0)}(p_k,
1) + \epsilon(r)
\end{align*}
from which we obtain a uniform lower bound for
$\vol_{\tilde{g}_k(0)}B_{\tilde{g}_k(0)}(p_k, 1)$. The
desired lower bound for the injectivity radius follows from
Cheeger's~lemma.

Since each $(M_k^n, \tilde{g}_k(\t))$ has a uniform
curvature bound and lower bound for the injectivity radius
at $(p_k, 0)$, according to Hamilton's compactness
theorem~\cite{Ha-Comp}, we can take a subsequence of
$\{(M_k^n, \tilde{g}_k(\t), p_k)\}_{k\in\Z}$ converging to
the limit Ricci flow $(M_\infty^n, g_\infty(\t), p_\infty),
\t \in [0, \alpha/2)$. From Lemma~\ref{limvol}, we infer
that $\tilV^{g_\infty}_{(p_\infty, 0)}(\bt) =1$, which
implies that the limit $(M_\infty^n, g_\infty(0))$ is
isometric to the Euclidean space by Theorem~\ref{reduced}.
This is in conflict with that $|\Rm|(p_\infty, 0)=1$. The
proof of Lemma~\ref{main2} is now complete.
\end{proof}

Now we present the proof of Theorem~\ref{main}.
\begin{proof}[Proof of Theorem~$\ref{main}$]
Take $\epsilon_n := \epsilon_n(3) >0$ from
Lemma~\ref{main2}. Suppose that $(M^n, g(\t)), \t\in [0,
\infty)$ is a complete ancient solution to the Ricci flow
with Ricci curvature bounded from below satisfying that
\begin{equation*}
\tilNu(g) > 1-\epsilon_n.
\end{equation*}
Due to Lemma~\ref{basept} and the monotonicity of the
reduced volume, we know that
\begin{equation*}
\tilV_{(p, \t)}(\bt) > 1-\epsilon_n \hbox{ for all } (p,
\t) \in M\times[0, \infty) \hbox{ and } \bt>0.
\end{equation*}
By Lemma~\ref{funda}, we know that $\pi_1(M)$ is finite,
 and applying Lemma~\ref{main2} for all $T>0$ yields that $(M^n, g(\t)), \t\in[0, \infty)$ is flat.
The only flat manifold with finite fundamental group is the
Euclidean space. Thus $(M^n, g(\t))$ is isometric to
$(\R^n, g_{\rm E})$ for all $\t\in[0, \infty)$, i.e.,
$(M^n, g(\t)), \t\in[0, \infty)$ is the Gaussian soliton.
This concludes the proof of Theorem~\ref{main}.
\end{proof}

\begin{rem}
Theorem~\ref{main} may have several variations. (See the
questions in \cite{Ni-Ent} for instance.) The following,
which also generalizes Theorem \ref{gap}, may be thought of
as one of them.
\end{rem}

\begin{thm}\label{main4}
There exists $\epsilon_n^{\prime}>0$ satisfying the
following: let $(M^n, g(\t)),$ $\t \in [0, \infty)$ be a
complete ancient solution to the Ricci flow with bounded
non-negative Ricci curvature. Suppose that the asymptotic
volume ratio $\nu(g(\t_0))$ of $g(\t_0)$ is greater than
$1-\epsilon_n^{\prime}$ for some $\t_0 \in[0, \infty)$.
Then $(M^n, g(\t)), \t\in [0, \infty)$ is the
Gaussian~soliton.
\end{thm}

The following proposition was proved by the author by
utilizing Cheeger--Colding's volume convergence theorem
\cite[Theorem 5.9]{ChCo} and Lemma~\ref{Pelem}(b).
\begin{prop}[{\rm \cite[Theorem 7]{Yo}}]
Let $(M, g(\t))$ be a complete backward Ricci flow with
bounded non-negative Ricci curvature. Then the asymptotic
volume ratio $\nu(g(\t))$ of $g(\t)$ is constant in~$\t$.
\end{prop}

The proof of Theorem~\ref{main4} is essentially the same as
that of Theorem~\ref{main} and we leave it to the
interested~reader.

We also comment here that Theorem~\ref{main4} is not true
when the ancient solution $g(\t)$ in the statement is
replaced with an {\it immortal solution} $g(t), t\in[0,
\infty)$ to the (forward) Ricci flow. In fact, one can show
that any Ricci flow $g(t), t\in [0, T)$ which has bounded
non-negative curvature operator and the initial metric
$g(0)=g_0$ with positive $\nu(g_0)>0$ extends to the
immortal solution $g(t), t\in[0, \infty)$. (See also the
example in \cite[Chapter 4, Section~5]{CLN}.)

\section{A gap theorem for gradient shrinkers}

In this section, we present the proof of
Corollary~\ref{main3} and discuss the case of expanding
Ricci solitons.

\subsection{Shrinking Ricci solitons}

We now prove Corollary~\ref{main3}. Recall that $(M^n, g,
f)$ is a complete gradient shrinking Ricci soliton with
Ricci curvature bounded below.
\begin{proof}[Proof of Corollary~$\ref{main3}$]
First, we construct an ancient solution to the Ricci flow.
(Recall the proof of Theorem~\ref{reduced}. See also
\cite[Theorem 4.1]{CLN}.) Define a one-parameter family of
diffeomorphisms $\phi_\t :M\to M, \t \in(0, \infty)$ by
\begin{equation*}
\frac{d}{d\t}\phi_\t = \frac{\lambda}{\t}\nabla f \circ
\phi_\t\quad \hbox{and}\quad \phi_\lambda= {\rm id}_M.
\end{equation*}
It is easy to see that the gradient vector field $\nabla f$
is complete, thanks to the assumption on the lower bound
for $\Ric$. Then we pull back $g$ by $\psi_\t :=
\phi_\t^{-1}$ so as to obtain a backward Ricci flow
$g_0(\t) := \frac{\t}{\lambda}(\psi_\t)^{*}g, \t \in(0,
\infty)$ with $g_0(\lambda)=g$. Put $g_{1}(\t) :=g(\t+1),
\t\in[0, \infty)$ and fix some point $p\in M$. It suffices
to show that
\begin{equation}\label{goal}
\tilNu(g_{1}) \ge \int_M (4\pi\lambda)^{-n/2}{\rm e}^{-f}\,d\mu_g
\end{equation}
since the left (resp. right)-hand side of (\ref{goal}) is
$\le 1$ (resp. $>$0).

Let us first give a heuristic argument. It seems reasonable
to hold that
\begin{equation*}
\tilV^{g_0}_{(p, 0)}(\t) = \tilNu(g_0) = \int_M
(4\pi\lambda)^{-n/2}{\rm e}^{-f}\,d\mu_g\quad \hbox{ for all }
\t>0
\end{equation*}
(cf. \cite{CHI}). Then inequality (\ref{goal}) will
follow from Lemma~\ref{basept}, if it is applicable to this
case. Of course, the problem arises from the fact that
$\t=0$ is the singular time for~$g_0(\t)$.

Now we give a rigorous proof. Recall that we have
normalized $f$ in~(\ref{fnormalize}) so that
\begin{equation*}
R_{g_0(\t)}+ |\nabla f_\t|_{g_0(\t)}^2
-\frac{f_\t}{\t}=0\quad \hbox{ for } \t>0
\end{equation*}
where $f_\t=f(\cdot, \t) := (\psi_\t)^* f = f\circ
\psi_\t$. Since $R_{g_0(\t)}$ is non-negative
(Proposition~\ref{ancient}), so is $f_\t$. Put $x_1
:=\phi_{\t_1}(x)$ and $x_2:=\phi_{\t_2}(x)$ for some $x\in
M$. Then it follows from the argument in \cite[p.~344]{V2}
that $\gamma(\t) :=\phi_{\t}\circ\phi^{-1}_{\t_1}(x_1)$ is
the $\L^{g_0}$-minimal geodesic from $(x_1, \t_1)$ to
$(x_2, \t_2)$ and
\begin{equation}\label{shrinkdist}
\frac{1}{2\sqrt{\t_2}} L^{g_0}_{(x_1, \t_1)}(x_2, \t_2) = f(x_2, \t_1) - \sqrt\frac{\t_1}{\t_2}f(x_1, \t_1).
\end{equation}

Fix a compact set $\K \subset M$, $\epsilon>0$ and
$\bt\gg1$. Take $q \in \phi_{\bt}(\K)$ and $p_2 \in
\phi_{2}(\K)$ with $q=\phi_{\bt}\circ \phi_{2}^{-1}(p_2)$.
From the triangle inequality for $\L$-distance,
Sublemma~\ref{timeshift} and (\ref{shrinkdist}), it follows
that
\begin{align*}
\l_{(p, 0)}^{g_1}(q, \bt-1)
&\le \frac{1}{2\sqrt{\bt-1}} \( L_{(p_2, 1)}^{g_1}(q, \bt-1) + L_{(p, 0)}^{g_1}(p_2, 1)\)\\
&\le \frac{1}{2\sqrt{\bt}} L_{(p_2, 2)}^{g_0}(q, \bt) + \frac{1}{2\sqrt{\bt-1}} \max_{\phi_{2}(\K)} L_{(p, 0)}^{g_1}(\cdot, 1)\\
&\le f(q, \bt) - \sqrt{\frac{2}{\bt}}f(p_2, 2) + C(\K)\bt^{-1/2}\\
&\le f(q, \bt) + C(\K)\bt^{-1/2}.
\end{align*}
From this, we deduce that
\begin{align*}
\tilNu(g_1)
&\ge \tilV_{(p, 0)}^{g_1}(\bt-1) -\epsilon\\
&\ge {\rm e}^{-C(\K)\bt^{-1/2}} \int_{\phi_\bt(\K)} (4\pi\bt)^{-n/2}{\rm e}^{-f(q, \bt)} d\mu_{g_0(\bt)}(q) -\epsilon\\
&= {\rm e}^{-C(\K)\bt^{-1/2}} \int_\K
(4\pi\lambda)^{-n/2}{\rm e}^{-f} d\mu_{g} -\epsilon.
\end{align*}
We have used the equation
\begin{equation*}
\int_M h\circ \psi_\bt\,d\mu_{(\psi_\bt)^*g} = \int_M
h\,d\mu_g\quad\hbox{for any } h\in L^1(d\mu_g)
\end{equation*}
which follows from the definition of pull back. Inequality
(\ref{goal}) then follows from the arbitrariness of
$\bt>0,\epsilon>0$ and $\K\subset M$.

By using~(\ref{goal}), Theorems~\ref{main} and
\ref{finitefundaRicci} immediately imply
Corollary~\ref{main3}. As for (3) of Corollary~\ref{main3},
it is easy to see that the Euclidean space regarded as a
shrinking soliton is the Gaussian soliton up to scaling
(cf. \cite[p.~416]{V2}). This concludes the proof of
Corollary~\ref{main3}.
\end{proof}

In the above proof, inequality (\ref{goal}) was enough for
our purpose, however, we can actually show that the
equality holds in (\ref{goal}) in the situation of
Corollary~\ref{main3}. Here we describe the proof of this
for future applications.
\begin{prop}\label{fvolredvol}
Let $(M^n, g, f)$ be a complete gradient shrinking Ricci soliton
with Ricci curvature bounded below by $-K \in\R$. Assume that
$f$ is normalized so that $(\ref{fnormalize})$ holds. Then, with
notation as in the proof of\break Corollary~\ref{main3}, we have
\begin{equation}
\tilNu(g_1) = \int_M (4\pi\lambda)^{-n/2}{\rm
e}^{-f}\,d\mu_g.
\end{equation}
\end{prop}

\begin{proof}
Take a sequence $\{\t_i\}_{i \in \Z}$ with $\t_i \to
\infty$ as $i \to \infty$ and put $\alpha_i :=
\sqrt{1-\frac{1}{\t_i}}$. Fix $r>0$ and $\bt>0$
sufficiently large. For any $q \in \K_i(\bt-1) := \L^{g_1}
K_{\bt-1, \t_i-1}(p, r)$, take $p_i :=\gamma(\t_i-1) \in
\K_i :=\L^{g_1} K_{\t_i-1, \t_i-1}(p, r)$, where $\gamma$
is the minimal $\L^{g_1}$-geodesic from $(p, 0)$ to~$(q,
\bt-1)$.

It follows from the combination of Sublemma~\ref{timeshift}
and (\ref{shrinkdist}) that
\begin{align*}
\l^{g_1}_{(p, 0)}(q, \bt-1)
&= \frac{1}{2\sqrt{\bt-1}} \( L^{g_1}_{(p_i, \t_i-1)}(q, \bt-1) + L^{g_1}_{(p, 0)}(p_i, \t_i-1) \)\\
&\ge \alpha_i \frac{1}{2\sqrt{\bt}} L^{g_0}_{(p_i, \t_i)}(q, \bt)\\
&= \alpha_i \(f(q, \bt)- \sqrt\frac{\t_i}{\bt}f(p_i, \t_i) \)\\
&\ge \alpha_i \( f(q, \bt)- \sqrt\frac{\t_i}{\bt} \max_{\K_i}f(\cdot, \t_i) \)\\
&= \alpha_i f(q, \bt) -C(\t_i)\bt^{-1/2}.
\end{align*}
Recall that $L^{g_1}_{(p, 0)}(\cdot, \cdot) \ge0$, which
follows from the non-negativity of the scalar curvature of
$g_1(\t)$ (Proposition~\ref{ancient}), and that $\K_i$ is
compact.

Thus, by Proposition \ref{epsilon},
\begin{align*}
\tilNu(g_1)&\le \(1-\frac{1}{\bt}\)^{n/2} \tilV_{(p, 0)}^{g_1}(\bt-1) + \epsilon(r)\\
&\le \int_{\K_i(\bt-1)} (4\pi\bt)^{-n/2}\exp\({-\l_{(p, 0)}^{g_1}(\cdot, \bt- 1)}\) d\mu_{g_0(\bt)} +3\epsilon(r)\\
&\le {\rm e}^{C(\t_i)\bt^{-1/2}} \int_{\K_i(\bt-1)} (4\pi\bt)^{-n/2}{\rm e}^{-\alpha_if(\cdot, \bt)}d\mu_{g_0(\bt)} +3\epsilon(r)\\
&\le {\rm e}^{C(\t_i)\bt^{-1/2}} \int_M (4\pi\lambda)^{-n/2}{\rm e}^{-\alpha_if}d\mu_{g} +3\epsilon(r).
\end{align*}

Now we observe that ${\rm e}^{-\alpha_if}$ is integrable
for large $i \in \Z$. To do this, let us recall that
$f$-volume $\int_M {\rm e}^{-f}d\mu$ is finite if the
Bakry--Emery tensor $\Ric+\Hess f$ is bounded below by
positive constant~\cite{M,WW}. In our case, we know that
\begin{equation*}
\Ric+\Hess \alpha_if \ge
\(\frac{\alpha_i}{2\lambda}-(1-\alpha_i)K\) g >0
\end{equation*}
and hence $\int_M {\rm e}^{-\alpha_if}d\mu_g$ makes sense
for all large $i \in \Z$.

Since $r>0$ and $\bt>0$ are arbitrary, we have obtained that
\begin{equation}\label{rvollefvol}
\tilNu(g_1) \le \int_M (4\pi\lambda)^{-n/2}{\rm e}^{-\alpha_if}d\mu_{g} \quad (< \infty)
\end{equation}
and the right-hand side of (\ref{rvollefvol}) converges to
the normalized $f$-volume as $i\to\infty$. Combined with
(\ref{goal}), this completes the proof of the proposition.
\end{proof}

\subsection{Expanding Ricci solitons}

Finally, we consider gradient expanders of non-negative
Ricci curvature and prove the result corresponding to
Corollary~\ref{main3} for them. A {\it gradient expanding
Ricci soliton} is a triple $(M, g, f)$ satisfying
\begin{equation*}
\Ric-\Hess f+\frac{1}{2\lambda}g=0
\end{equation*}
for some positive constant $\lambda>0$. We normalize $f \in
C^\infty(M)$ so that $R+ |\nabla f|^2-\lambda^{-1}f=0$ on
$M$ for the expander $(M, g, f)$~too.

\begin{prop}[{\rm \cite{CN}}]\label{CN}
Let $(M^n, g, f)$ be a complete expanding Ricci soliton
with non-negative Ricci curvature. Then
\begin{enumerate}
\item[(1)]$M^n$ is diffeomorphic to $\R^n$.
\item[(2)]We have
\begin{equation}\label{eq:expandfvol}
\int_M (4\pi\lambda)^{-n/2}{\rm e}^{-f}\,d\mu_g \le 1
\end{equation}
and the equality holds if and only if $(M^n, g, f)$ is, up
to scaling, the expanding Gaussian soliton $(\R^n, g_{\rm
E}, \frac{|\,\cdot\,|^2}{4})$.
\end{enumerate}
\end{prop}

We remark that the proposition is a restatement of a result
of~\cite{CN}. Because our proof is simple and purely
geometric in contrast to the one in~\cite{CN}, we decided
to include it here.

\begin{proof}[Proof of Proposition~\ref{CN}]
First, we note that the potential function $f$ is bounded
below and $\frac{1}{2\lambda}$-convex, i.e., $\Hess f \ge
\frac{1}{2\lambda}g>0$. Therefore, $f$ has the unique
critical point $p\in M$ where the minimum value of $f$ is
attained. Part (1) of the proposition follows from this.

Next, as in the proof of Corollary~\ref{main3}, we
construct a self-similar solution to the (forward) Ricci
flow $g_0(t):=\frac{t}{\lambda}(\psi_t)^{*}g, t\in(0,
\infty)$ and put $g_1(t) := g_0(t+1), t\in [0, \infty)$.

Define the {\it forward reduced distance} at $(q, \btt) \in M \times (0, \infty)$ by
\begin{equation*}
\l_{(p, 0)}^{+}(q, \btt) := \frac{1}{2\sqrt{\btt}} \inf
\bigl\{ \L^{+}(\gamma) ;\  \gamma(0)=p, \gamma(\btt)=q
\bigr\}
\end{equation*}
where we defined the {\it forward $\L$-length}
$\L^{+}(\gamma)$ of $\gamma : [0, \btt]\to M$ by
\begin{equation*}
\L^{+}(\gamma) := \int_0^{\btt} \sqrt{t}
\(\|\frac{d\gamma}{dt}\|^2_{g_1(t)} +R_{g_1(t)}(\gamma(t))
\) dt.
\end{equation*}
Then we consider the formal reduced volume defined by
\begin{equation}\label{formalforward}
\widehat{V}_{(p, 0)}^{g_1}(t) := \int_M (4\pi t)^{-n/2}{\rm e}^{-\l^{+}_{(p, 0)}(\cdot, t)}d\mu_{g_1(t)}.
\end{equation}
We do not care whether $\widehat{V}^{g_1}_{(p, 0)}(t)$ is
monotone. (This is the case when $g_1(t)$ has bounded
non-negative curvature operator or non-negative bi-sectional
curvature in the K\"{a}hler case~\cite{Ni-LYH}.)

Since $(M^n, g_1(t))$ has non-negative Ricci curvature, we
have
\begin{align*}
\l^{+}_{(p, 0)}(q, \btt-1)&\ge \frac{1}{2\sqrt{\btt-1}} \inf_{\gamma} \int_0^{\btt-1} \sqrt{t}\|
\frac{d\gamma}{dt}\|_{g_1(\btt-1)}^2 dt\\
&= \frac{1}{4(\btt-1)}d_{g_1(\btt-1)}(p, q)^2 \ge
\frac{1}{4\lambda}d_g(p, \psi_{\btt}(q))^2
\end{align*}
and by Lemma~\ref{Ricciflat},
\begin{align*}
\(\frac{\btt-1}{\btt}\)^{n/2}\widehat{V}_{(p, 0)}^{g_1}(\btt-1)
&\le \int_M(4\pi \btt)^{-n/2}\exp\({-\frac{1}{4\lambda}d_g(p, \psi_{\btt}(\cdot))^2}\) d\mu_{g_0(\btt)}\\
&= \int_M (4\pi\lambda)^{-n/2}\exp\({-\frac{d_g(p, \cdot)^2}{4\lambda}}\) d\mu_g\\
&\le 1.
\end{align*}

Then, from the same argument as in the derivation
of~(\ref{goal}) in the proof of Corollary~\ref{main3}, we
derive that
\begin{equation*}
\int_M (4\pi\lambda)^{-n/2}{\rm e}^{-f}d\mu_g \le
\liminf_{t\to\infty} \widehat{V}_{(p, 0)}^{g_1}(t)
\end{equation*}
and hence
\begin{equation}\label{forwardreduced}
\int_M (4\pi\lambda)^{-n/2}{\rm e}^{-f}d\mu_g \le \int_M
(4\pi\lambda)^{-n/2}\exp\({-\frac{d_g(p,
\cdot)^2}{4\lambda}}\) d\mu_g \le 1
\end{equation}
which yields~(\ref{eq:expandfvol}).

When the normalized $f$-volume is $1$, we have equalities
in (\ref{forwardreduced}). Then we know from the equality
case of Lemma~\ref{Ricciflat} that $(M^n, g)$ is isometric
to the Euclidean space. The only way to regard $(\R^n,
g_{\rm E})$ as a gradient expanding Ricci soliton is the Gaussian
soliton, up to rescaling. This finishes the~proof.
\end{proof}

\section{Concluding remarks}

In this section, we collect some remarks.
\begin{rem}\label{rem:entropy}
Let $(M^n, g(\t)), \t\in [0, T)$ be a super Ricci flow
$\dertau{}g =:2h \le 2\,\Ric$ satisfying
Assumption~\ref{as} on a closed manifold $M$. Put $H:= \tr
h$. Following~\cite{P}, we define the {\it
$\mathcal{W}$-entropy} for a triple $(g(\t), f, \t)$ by
\begin{equation}
\mathcal{W}(g(\t), f, \t) = \int_M \[\t(|\nabla f|^2 +H) +
f-n\]u\,d\mu_{g(\t)}
\end{equation}
where $f$ is a smooth function on $M^n$, $\t>0$ and
$u:=(4\pi\t)^{-n/2}{\rm e}^{-f}$.

We evolve $u$ by the conjugate heat equation $\dertau{}u =
\varDelta_{g(\t)} u - H u$, or equivalently,
\begin{equation*}
\dertau{f} = \varDelta_{g(\t)} f -|\nabla f|^2 + H
-\frac{n}{2\t}.
\end{equation*}

Then, by simple calculation, we obtain the entropy formula
for the super Ricci flow:
\begin{align*}
&{d\over d\t} \mathcal{W}(g(\t), f, \t)\\
&\quad=-2\t \int_M \biggl[\Big|h +\Hess f -\frac{1}{2\t}g\Big|^2+(dH-2{\rm div}h)(\nabla f)\\
&\qquad + (\Ric -h)(\nabla f, \nabla f)- \frac{1}{2}\(
\dertau{H} + \varDelta_{g(\t)} H + 2|h|^2\) \biggl] u\,
d\mu_{g(\t)}\\
&\quad\le 0
\end{align*}
from which we simultaneously recover the entropy formulae
of Perelman $(h=\Ric)$~\cite{P} and Ni
$(h=0)$~\cite{Ni-Ent}.

We also have similar formula for the super Ricci flow analogue
of\break {\it $\mathcal{F}$-entropy} introduced in
\cite[Section~1]{P}.
\end{rem}

\begin{rem}
(1) We can find the optimal value $\epsilon_n$ of the constant
obtained in Theorem~\ref{main}, namely $\epsilon_n := 1-
\max\{\tilNu(g)\} >0$. We take the maximum over all the complete
$n$-dimensional non-Gaussian ancient solutions to the Ricci flow
with Ricci curvature bounded below. The maximum is achieved, as
is seen by the limit argument used in the proof of
Lemma~\ref{main2}. Then it is easy to see that
$\{\epsilon_n\}_{n=2}^\infty$ is a non-increasing sequence. It
seems interesting to determine the exact value of $\lim_{n\to
\infty} \epsilon_n$.

(2) Now we calculate an asymptotic reduced volume (or
normalized\break $f$-volume) for the round $n$-sphere $(S^n,
g_{S^n})$ with constant Ricci curvature
$\Ric=\frac{1}{2}g_{S^n}$. Then $g(\t) := (1+\t)g_{S^n},
\t\in[0, \infty)$ is an ancient solution to the Ricci flow,
while $(S^n, g_{S^n}, f)$ with $f \equiv \frac{n}{2}$ is a
gradient shrinking Ricci soliton. By
Proposition~\ref{fvolredvol},
\begin{align*}
\tilNu(g)&= \int_{S^n} (4\pi)^{-n/2}{\rm e}^{-n/2}d\mu_{g_{S^n}}\\
&= \frac{\sqrt{2\pi}m^{m+\frac{1}{2}} {\rm
e}^{-m}}{\Gamma(m+1)} \sqrt{\frac{2}{e}} \nearrow
\sqrt{\frac{2}{e}}\quad \hbox{as } n\nearrow \infty.
\end{align*}
Here we have put $n=2m+1, m\in \frac{1}{2}\Z$
and used that
$\vol(S^n, \frac{1}{2(n-1)}g_{S^n}) = 2\pi^{m+1}/\Gamma(m+1)$ and Stirling's formula:
\begin{equation*}
\Gamma(m+1) = \sqrt{2\pi} m^{m+\frac{1}{2}}{\rm e}^{-m}
{\rm e}^{\theta(m)} \hbox{ for }m>0,
\end{equation*}
where $\theta(m) \searrow 0$ as $m\nearrow \infty$. This
gives an upper bound for the constant $\epsilon_n$ obtained
in Theorem~\ref{main}:
\begin{equation*}
\epsilon_n \le 1- {\rm e}^{-\theta(m)} \sqrt{2{\rm e}^{-1}} \searrow
1-\sqrt{2{\rm e}^{-1}} \hbox{ as } n\nearrow \infty.
\end{equation*}
\end{rem}

\begin{rem}\label{rem:funda}
Theorem~\ref{finitefundaRicci} has another corollary which
was pointed out by Professor Lei Ni.
\end{rem}
\begin{cor}\label{cor:blowupfunda}
Let $(M^n, g(t)), t\in[0, T)$ be a complete Ricci flow
with\break bounded curvature and positive injectivity radius at
$t=0$ which develops singularity at finite time $t=T<\infty$.
Then any singularity model of $(M^n, g(t))$ has finite
fundamental~group.
\end{cor}
The singularity model is the limit of dilations of $(M^n, g(t))$
around a singular point (see \cite[Chapter 8]{CLN} for the
precise definition). We can take such a blow-up limit in the
corollary by virtue of Perelman's no local collapsing theorem
\cite[Section~7]{P} and Hamiton's compactness
theorem~\cite{Ha-Comp}.\break The corollary immediately follows
from the fact that such a singularity\break model is an ancient
solution with positive asymptotic reduced volume\break (cf.
\cite[Lemma 8.22]{V2}).

We will be able to use this corollary in order to
understand the singularities of the Ricci flow further. For
example, we can prove the following: for any ancient
solution $(N^{n-1}, g_N(t)), t\in(-\infty, \alpha)$, the
canonical ancient solution on $S^1 \times N^{n-1}$ cannot
occur as a blow-up limit of the Ricci flow as in
Corollary~\ref{cor:blowupfunda}. In the case where $N$ is a
sphere, this result was conjectured by Hamilton
\cite[Section~26]{Ha-Form} and proved by
Ilmanen--Knopf~\cite{IK}.

\begin{rem}
(1) Feldman--Ilmanen--Ni \cite{FIN} have discovered the
{\it forward reduced volume} $\tilV^{+}_{(p, 0)}(t)$ for
the (forward) Ricci flow $(M^n, g(t)), t\in[0, T)$ which is
non-increasing in $t$. However, its definition is given by
\begin{equation*}
\tilV^{+}_{(p, 0)}(t) := \int_M (4\pi t)^{-n/2}{\rm
e}^{\l^{+}_{(p, 0)}(\cdot, t)}\,d\mu_{g(t)}
\end{equation*}
(cf. with (\ref{formalforward})) and it is not well defined
for general non-compact manifolds. It is not likely that
Theorem~\ref{main} has an analogue for the forward reduced
volume $\tilV^{+}_{(p, 0)}(t)$.

(2) One can also easily generalize the monotonicity of
$\tilV^{+}_{(p, 0)}(t)$ to the forward super Ricci flows
$\dert{}g\ge -2\,\Ric$, if the condition corresponding to
Assumption~\ref{as} is imposed.
\end{rem}

\begin{rem}\label{finalrem}
In Carrillo--Ni's preprint~\cite{CN}, the potential
function $f$ of the gradient Ricci soliton $(M^n, g, f)$ is
normalized so that
\begin{equation*}
\int_M (4\pi\lambda)^{-n/2}{\rm e}^{-f}d\mu_g =1.
\end{equation*}
Then their main result is the logarithmic Sobolev
inequality for gradient Ricci solitons with $\mu(g, f) :=
\lambda(R+|\nabla f|^2)-f$ as the best constant. They also
showed that $\mu(g, f) \ge0$ for gradient shrinking Ricci
solitons (under the curvature condition stronger than ours)
and conjectured that $\mu(g, f)=0$ implies that it is the
Gaussian soliton. It is easily checked that $\mu(g, f)=
-\log\vol_f(M)$, where $\vol_f(M)$ is the normalized
$f$-volume of $(M^n, g, f)$ with $f$ being normalized in
our sense as in (\ref{fnormalize}). Hence,
Corollary~\ref{main3}.(3) gives an affirmative answer to
the conjecture in~\cite{CN}.
\end{rem}

\begin{rem}
After the first version of this paper was completed, the
result of Zhang~\cite{Z} came to the author's attention. It
states that for any gradient Ricci soliton $(M, g, f)$, the
completeness of $g$ implies that of $\nabla f$. Recall that
we have used the assumption that $\Ric \ge -K$ for some
$K\in \R$ in the proof of Corollary~\ref{main3} only to
ensure the completeness of $\nabla f$ and the existence of
minimal $\L$-geodesics between any two points in
space--time. A natural question is whether the assumption
on $\Ric$ in the statement of Corollary~\ref{main3} is
superfluous.
\end{rem}

\appendix

\def\thesection{A}\setcounter{section}{0}
\section*{Appendix}

In this appendix, we present very detailed proofs to the
facts on the super Ricci flow used in the proof of main
theorem. The proofs rely on the following lemma whose proof
in~\cite{P} works as well for the super Ricci flow.

\begin{lem}[{\rm \cite[Lemma 8.3]{P}}]\label{Pelem}
Let $(M^n, g(\tau))$ be a complete super Ricci flow.
\begin{enumerate}
\item[(a)]
Assume that $\Ric(\cdot, \t_0) \le (n-1)K$ on the ball $B_{\t_0}(x_0, r_0)$.
Then outside of $B_{\t_0}(x_0, r_0)$,
\begin{equation*}
\(\dertau{}+\varDelta_{g(\t_0)}\) d_{\t_0}(\cdot, x_0) \le
(n-1)\(\frac{2}{3} Kr_0+r_0^{-1}\).
\end{equation*}
The inequality is understood in the barrier sense.
\item[(b)]
Assume that $\Ric(\cdot, \t_0) \le (n-1)K$ on the union of the balls $B_{\t_0}(x_0, r_0)$ and $B_{\t_0}(x_1, r_0)$.
Then
\begin{equation*}
\frac{d^+}{d\t} d_{\t}(x_0, x_1)\|_{\t=\t_0} \le
2(n-1)\(\frac{2}{3}Kr_0 + r_0^{-1}\).
\end{equation*}
Here, $\frac{d^+}{d\t}f(\t) :=
\limsup_{\epsilon\to0+}\frac{f(\t+\epsilon)-f(\t)}{\epsilon}$
denotes the upper Dini\break derivative.
\end{enumerate}
\end{lem}

\subsection{Perelman's Point picking lemma}

\begin{lem}[{\rm [25, Section 10; 17, Lemmas 30.1, 31.1]}]\label{pointpick}
Let $(M^n, g(\t)), \t \in[0, T)$ be a complete super Ricci
flow and $A, B >0$ are arbitrary numbers. Assume that there
exists a point $(x_1, \t_1)\in M(B)$, where $M(B) := \{ (x,
\t) \in M\kern-1pt \times\kern-1pt [0, T) ; |\Rm|(x, \t)(T\kern-1pt-\kern-1pt \t)\kern-1pt >\kern-1pt B \}$. Then
we can find a point $(p_*, \t_*)\kern-1pt \in\kern-1pt M(B)$ such that
\setcounter{equation}{0}
\begin{equation}\label{le2}
|\Rm|(x, \t) \le 2|\Rm|(p_*, \t_*) =: 2Q
\end{equation}
for all $(x, \t)$ with
$d_{\t_*}(x, p_*) < AQ^{-1/2}$ and $\t_* \le \t \le \t_* + \frac{1}{2}BQ^{-1}$.
\end{lem}

The proof is divided into two steps as in~\cite{P}.

\begin{clm}
Take $x_0\in M$ and $A^\prime>0$ satisfying that
$4(n-1)B\epsilon \le 1 \hbox{ and}\break (\epsilon A^\prime)^2
\ge 3/2 \hbox{ for some small } \epsilon>0 \hbox{ and }
A^\prime \ge 2A.$ Then we can find a point $(p_*, \t_*)\in
M(B)$ such that $(\ref{le2})$ holds for all $(x, \t)$ with
\begin{equation*}
d_{\t}(x, x_0) < d_{\t_*}(p_*, x_0) + A^\prime Q^{-1/2}
\hbox{ and  }\t_* \le \t \le \t_*+ \tfrac{1}{2}BQ^{-1}.
\end{equation*}
\end{clm}

\begin{proof}
If not, we can construct a sequence $\{(x_i, \t_i)\}_{i\in
\Z} \subset M(B)$ starting from $(x_1, \t_1) \in M(B)$
satisfying that
\begin{equation*}
Q_{i+1} > 2Q_i,\quad d_{i+1} < d_i + A^\prime
Q_i^{-1/2}\quad \hbox{and}\quad \t_i \le \t_{i+1} \le \t_i
+\tfrac{1}{2}BQ_i^{-1}
\end{equation*}
where we put $Q_i := |\Rm|(x_i, \t_i)$ and $d_i:=
d_{\t_i}(x_i, x_0)$. We see that $(x_{i+1}, \t_{i+1})$ lies
in $M(B)$ if $(x_i, \t_i)$ does. Indeed,
\begin{align*}
Q_{i+1}(T-\t_{i+1}) -B &> 2Q_i \( T- \t_i - \tfrac{1}{2}BQ_i^{-1}\) -B\\
&= 2\( Q_i(T-\t_i)- B \) >0.
\end{align*}
This implies that $\{(x_i, \t_i)\}_{i\in\Z} \subset M(B)$.

Then $Q_i > 2^{i-1}Q_1 \to \infty$ as $i \to\infty$, which
contradicts to that
\begin{equation*}
\t_i \le \t_1+ BQ_1^{-1} < T -\epsilon_1 <T\quad
\hbox{and}\quad d_i \le d_1 + 2A^\prime Q_1^{-1/2}.
\end{equation*}
Here $\epsilon_1>0$ is taken so that $|\Rm|(x_1, \t_1)(T-
\t_1 -\epsilon_1)
>B$. Hence the\break sequence $\{(x_i, \t_i)\}$ stops at finite
steps and the terminal one is the desired\break point~$(p_*,
\t_*)$.
\end{proof}

\begin{clm}
The point $(p_*, \t_*)$ just obtained satisfies the desired
property.
\end{clm}

\begin{proof}
Take $x \in B_{\t_*}(p_*, AQ^{-1/2})$ and put $r_0 :=
\epsilon A^\prime Q^{-1/2}$. Let $\t^{\prime} \in[\t_*,
\t_*+ \frac{1}{2}BQ^{-1}]$ be the supremum of
$\t^{\prime\prime}$ such that
\begin{equation*}
|\Rm|(\cdot, \t) \le 2Q \hbox{ on } B_\t(x_0, r_0) \cup
B_\t(x, r_0) \hbox{ for all } \t \in[\t_0,
\t^{\prime\prime}].
\end{equation*}
It follows easily from the choice of $(p_*, \t_*)$ that
$\t^{\prime} > \t_*$ and $|\Rm| \le 2Q$ on $B_\t(x_0, r_0)$
for $\t \in [\t_*, \t_*+ \frac{1}{2}BQ^{-1}]$.

Applying Lemma~\ref{Pelem}(b) for $r_0 = \epsilon A^\prime
Q^{-1/2}$,
\begin{align*}
d_{\t^{\prime}}(x, x_0) - d_{\t_*}(x, x_0)
&\le 2(n-1)\(\tfrac{4}{3}\epsilon A^\prime Q^{1/2} +(\epsilon A^\prime)^{-1}Q^{1/2}\)(\t^{\prime}-\t_*)\\
&\le 2(n-1) \epsilon A^\prime BQ^{-1/2}\\
&\le \tfrac{1}{2}A^\prime Q^{-1/2}.
\end{align*}

Therefore, we have that
\begin{equation*}
d_{\t^{\prime}}(x, x_0) \le d_{\t_*}(x, p_*) +
d_{\t_*}(p_*, x_0) + \tfrac{1}{2}A^\prime Q^{-1/2} <
d_{\t_*}(p_*, x_0) + A^\prime Q^{-1/2}
\end{equation*}
and $\t^{\prime}= \t_* +\frac{1}{2}BQ^{-1}$. As $x \in
B_{\t_*}(p_*, AQ^{-1/2})$ is arbitrary, we conclude that
\begin{equation*}
|\Rm| \le 2Q \hbox{ on } B_{\t_*}(p_*, AQ^{-1/2})\times
[\t_*, \t_* +\tfrac{1}{2}BQ^{-1}].
\end{equation*}
This completes the proof of the lemma.
\end{proof}

\subsection{Ancient solutions have non-negative scalar curvature}

\begin{prop}[{\rm \cite[Proposition 2.1]{Chen}}]\label{ancient}
Any complete ancient super Ricci flow $(M^n, g(\t)), \t\in [0, \infty)$
satisfying $(\ref{evol})$ has non-negative trace of time derivative $2H := \tr \dertau{}g \ge0$.
\end{prop}

Note that we have no assumption on the bound of
$\dertau{}g$ in Proposition~\ref{ancient}.

\begin{proof}
We give a proof by contradiction which is based on the
maximum principle argument. Assume that $H(x_0, \t_0)\kern-1pt<\kern-1pt 0$
for some $(x_0, \t_0)\kern-1.2pt\in\kern-1.2pt M\kern-1.2pt\times\kern-1.2pt[0, \infty)$. We may assume
that $\t_0=0$.

Let $\phi : \R \to [0, 1]$ is a non-increasing $C^2$
function satisfying that $\phi =1$ on $(-\infty, 1/2]$,
$\phi = 0$ on $[1, \infty)$ and $\phi^{\prime\prime}-
\frac{2\phi^{\prime2}}{\phi} \ge -C\sqrt\phi$ on $(-\infty,
1)$ for some $C>0$. Such a function can be constructed from
$\phi(s)=(s-1)^4$ for $s\le1$ near $s=1$.

Take sufficiently large $T_0>0$ so that $n|H|(x_0, 0)^{-1}
\le T_0$. Find $r_0>0$ such that $\Ric \le (n-1)r_0^{-2}$
on $B_\t(x_0, r_0)$ for all $\t\in [0, T_0]$ and fix $A>0$
so large enough that $|H|(x_0, 0) \ge nC(Ar_0)^{-2}$.

Put
\begin{equation*}
u(x, \t) := \phi\( \frac{d_\t(x, x_0) -
\frac{5}{3}(n-1)r_0^{-1}\t}{Ar_0} \) H(x, \t).
\end{equation*}
Then
\begin{align*}
\(\dertau{}+\varDelta\) u(x, \t)&= \phi^\prime \frac{(\dertau{}+\varDelta)d_{\t}(x, x_0) -\frac{5}{3}(n-1)r_0^{-1}}{Ar_0}H(x, \t)\\
&\quad+\phi\( \dertau{}+\varDelta \) H(x, \t)
+\phi^{\prime\prime}\frac{H(x, \t)}{(Ar_0)^2} +2\langle
\nabla\phi, \nabla H\rangle(x, \t).
\end{align*}

Let $u_{\min}(\t) := \min_{x\in M}u(x, \t)$ and assume that
$u_{\min}(\t_1) = u(x_1, \t_1) <0$ for some $\t_1\ge0$ and
$x_1\in M$. Then $d_{\t_1}(x_1, x_0) < Ar_0
+\frac{5}{3}(n-1)r_0^{-1}\t_1$ and $H(x_1, \t_1)<0$.
Furthermore, we have $\nabla u(x_1, \t_1)= 0$ and
$\varDelta u(x_1, \t_1) \ge0$.

If $d_{\t_1}(x_1, x_0)<r_0$, then $u=H$ near $(x_1, \t_1)$ and
\begin{align*}
\frac{d^{+}}{d\t}u_{\min}(\t_1)
&\le \liminf_{\t \searrow \t_1} \frac{u(x_1, \t)- u(x_1, \t_1)}{\t-\t_1}\\
&\le -\varDelta H(x_1, \t_1) -2|h|^2(x_1, \t_1)\\
&\le -\frac{2}{n}H^2(x_1, \t_1)=
-\frac{2}{n}u_{\min}(\t_1)^2.
\end{align*}

If $d_{\t_1}(x_1, x_0) \ge r_0$, by Lemma~\ref{Pelem}(a)
and that $2ab \le a^2+b^2$,
\begin{align*}
\frac{d^{+}}{d\t}u_{\min}(\t_1)
&\le -2\phi |h|^2(x_1, \t_1) +\(\phi^{\prime\prime}-\frac{2\phi^{\prime2}}{\phi}\) \frac{H(x_1, \t_1)}{(Ar_0)^2}\\
&\le -\frac{2}{n}\phi H^2(x_1, \t_1) -C\sqrt\phi \frac{H(x_1, \t_1)}{(Ar_0)^2}\\
&\le -\frac{2}{n}\phi H^2(x_1, \t_1)  + \frac{nC^2}{2(Ar_0)^4} +\frac{1}{2n}\phi H^2(x_1, \t_1)\\
&= -\frac{1}{n}u_{\min}(\t_1)^2  + \frac{nC^2}{2(Ar_0)^4} -\frac{1}{2n}u_{\min}(\t_1)^2.
\end{align*}

Since $|u_{\min}|(0) \ge |H|(x_0, 0) \ge nC(Ar_0)^{-2}$, we
know that $\frac{d^{+}}{d\t} u_{\min} \le
-\frac{1}{n}\break u_{\min}^2$ on $[0, T_0]$. Therefore,
\begin{equation*}
u_{\min}(\t) \le \frac{n}{nu_{\min}(0)^{-1}+\t}
\longrightarrow -\infty
\end{equation*}
as $\t \to n|u_{\min}|(0)^{-1} \le T_0$. This is the
desired contradiction.
\end{proof}

\section*{Acknowledgments}

The author expresses his gratitude to Professors Takao
Yamaguchi and Koichi Nagano for their comments and
discussions. He also would like to thank Professor Lei Ni
for his interest in this work and insightful comments. This
work was partially supported by Research Fellowships of the
Japan Society for the Promotion of Science for Young
Scientists.

\noindent {\sc Graduate School of Pure and Applied Sciences\\
University of Tsukuba\\
305-8571 Tsukuba\\ Japan}\\
\textit{E-mail address}: {\tt takumiy@math.tsukuba.ac.jp}
\end{document}